\documentclass[11pt]{amsart}
\usepackage{amsmath,amsfonts,amsthm,amssymb,mathrsfs}
\usepackage{epsfig}
\usepackage{slashed}
\usepackage{mathtools}
\usepackage{enumerate}
\usepackage{fullpage}
 \usepackage{soul}
\usepackage{xcolor}
\usepackage{float}
\usepackage{hyperref}
\usepackage{cleveref}

\newcommand{\mb}{\mathbf}
\newcommand{\mc}{\mathcal}

\renewcommand{\Re}{\mathrm{Re}\,}

\newcommand{\rg}{\mathrm{rg}\,}
\newcommand{\N}{\mathbb{N}}
\newcommand{\R}{\mathbb{R}}

\newcommand{\B}{\mathbb{B}}

\DeclareMathOperator{\range}{rg}

\renewcommand{\r}{\rho}
\renewcommand{\d}{\delta}
\renewcommand{\t}{\tau}

\newtheorem{lemma}{Lemma}[section]
\newtheorem{theorem}[lemma]{Theorem}
\newtheorem{corollary}[lemma]{Corollary}
\newtheorem{proposition}[lemma]{Proposition}
\theoremstyle{remark}
\newtheorem{remark}[lemma]{Remark}

\theoremstyle{definition}

\numberwithin{equation}{section}

\title[]{Singularity formation for the higher dimensional Skyrme model}

\author{Michael McNulty}
\address{Michigan State University, 619 Red Cedar Road, East Lansing, MI 48824}
\email{mcnult50@msu.edu}

\begin{document}
\begin{abstract}
This paper demonstrates that singularities form in the classical $(5+1)$-dimensional, co-rotational Skyrme model. It was recently proven by Chen, Sch\"orkhuber, and the author that the strong field limit of the $(5+1)$-dimensional, co-rotational Skyrme model admits an explicit self-similar solution which is asymptotically stable within backwards light cones. Seeded by the limiting model, we construct an open set of initial data whose evolution within a backwards light cone, according to the full model, suffers a gradient blowup in finite time. Moreover, the singularity develops at the self-similar rate and possesses an asymptotic profile given by the self-similar profile of the strong field model.
\end{abstract} 

\maketitle

\section{Introduction}\label{Introduction}
	We consider a natural generalization of the Skyrme model to higher dimensions. For $d\geq2$, and maps $\Psi:\mathbb R^{1+d}\to\mathbb S^d$, the (generalized) Skyrme model is described by the action functional
		\begin{equation}
			\mc S_{Sky}[\Psi]=\alpha \mc S_{WM}[\Psi]+\frac{\beta}{4}\int_{\R^{1+d}}\Big(\big(\eta^{\mu\nu}(\Psi^*h)_{\mu\nu}\big)^2-(\Psi^*h)_{\mu\nu}(\Psi^*h)^{\mu\nu}\Big)d\eta \label{skyrme action}
		\end{equation}
	where $\alpha, \beta \geq 0$, $\eta=\text{diag}(-1,1,\dots,1)$ denotes the Minkowski metric, $h$ is the standard round metric on $\mathbb S^d$, $(\Psi^*h)_{\mu\nu}=h_{ab}(\Psi)\partial_\mu \Psi^a\partial_\nu \Psi^b$ for\footnote{We are using the Einstein summation convention of implicitly summing over repeated lower and upper indices.} $\mu,\nu=0,\dots,d$ and $a,b=1,\dots,d$, and 
		\begin{equation}
			\mc S_{WM}[\Psi]=\frac{1}{2}\int_{\R^{1+d}}\eta^{\mu\nu}(\Psi^*h)_{\mu\nu}d\eta \label{wavemaps action}
		\end{equation}
	is the classical wave maps action describing the nonlinear sigma-model in the special case $d=3$. In particular, we restrict our attention to \textit{co-rotational} (one-equivariant) maps. These are maps $\Psi$ which respect the rotational symmetry of the domain and target in the sense that a rotation in the domain is mapped to the same rotation in the image. Expressed in spherical coordinates on their domain and co-domain, co-rotational maps take the form
		$$
			\Psi(t,r,\omega)=\big(\psi(t,r),\omega\big)
		$$
	for some function $\psi:\R \times[0,\infty)\to \R$ corresponding to a polar angle on $\mathbb S^d$ and $\omega \in \mathbb S^{d-1}$. For such maps, the Euler-Lagrange equation for \eqref{skyrme action} reduces to the following radial quasilinear wave equation
		\begin{align}
		\begin{split}
			\Big(\alpha&+\frac{\beta(d-1)\sin^2(\psi)}{r^2}\Big)\big(\partial_t^2\psi-\partial_r^2\psi\big)-\frac{d-1}{r}\Big(\alpha+\frac{\beta(d-3)\sin^2(\psi)}{r^2}\Big)\partial_r\psi
			\\
			&+\frac{(d-1)\sin(2\psi)}{2r^2}\bigg(\alpha+\beta\Big(\big(\partial_t\psi\big)^2-\big(\partial_r\psi\big)^2+\frac{(d-2)\sin^2(\psi)}{r^2}\Big)\bigg)=0. \label{skyrme eqn}
		\end{split}
		\end{align}
	We refer the reader to Appendix A of \cite{CMS23} for the details of its derivation. 
	
	In the 1960s, there was significant interest in modeling the dynamics of elementary particles via geometric field theories such as \eqref{skyrme action} and \eqref{wavemaps action}. Tony Skyrme established his namesake model in nuclear physics \cite{S61a,S61b,S62} by introducing a higher-order correction term to the previously well-established nonlinear sigma-model for pions \cite{GL60}. The intent of this correction term was to prevent the formation of singularities in the nonlinear sigma model and to allow for the existence of particle-like solutions, i.e., solitons. At the time, a heuristic relating scale invariance of the nonlinear sigma model's equation of motion and its conserved energy's behavior under rescaling suggested that particle-like solutions would shrink to a point. The existence of self-similar solutions for the nonlinear sigma model was rigorously proven by Shatah \cite{S88}. 
	
	Skyrme's correction to the nonlinear sigma model breaks scale invariance, thereby preventing the existence of \textit{exactly} self-similar solutions. Moreover, its conserved energy's behavior under rescaling (see Section \ref{The strong field limit of the Skyrme model}) suggests that, at least for the original three-dimensional model, soliton solutions should not shrink to a point. By now, it is rigorously known that Skyrme's correction term indeed succeeds in its original goal. Namely, the three-dimensional model admits a topological soliton - the \textit{Skyrmion} - and large data for the associated Cauchy problem yield globally regular solutions. The former was proven by Kapitanskii and Ladyzhenskaya \cite{KL83} as well as McLeod and Troy \cite{MT91}. The latter was first proven by Li \cite{L21}. The regularity of the initial data in Li's result was then improved upon by Geba and Grillakis \cite{GG18}. See Section \ref{Discussion of related results and history of the Skyrme model} for a detailed summary of these and similar results. 
	
	To the best of the author's knowledge, nothing appears to be known concerning the long-time dynamics of solutions of Equation \eqref{skyrme eqn} in dimensions $d\geq4$. The present work provides the first such result. Namely, for $d=5$, we prove that there exists an open set of regular initial data for the Cauchy problem associated to Equation \eqref{skyrme eqn}  whose solution exists within a backwards light cone and suffers a gradient blowup in finite time. Moreover, the asymptotic profile of the solution is given by an explicit, stable, self-similar solution of Equation \eqref{skyrme eqn} in the limiting case $\alpha=0$.

	\subsection{The strong field limit of the Skyrme model}\label{The strong field limit of the Skyrme model}
		A direct calculation verifies that solutions of Equation \eqref{skyrme eqn} formally conserve the energy-type quantity
			$$
				E[\psi]=\int_0^\infty\bigg(\Big(\alpha+\frac{\beta(d-1)\sin ^2(\psi)}{r^2}\Big)\frac{\big(\partial_t\psi\big)^2+\big(\partial_r\psi\big)^2}{2}+\frac{d-1}{2}\Big(\alpha+\frac{\beta(d-2)\sin^2(\psi)}{2r^2}\Big)\frac{\sin^2(\psi)}{r^2}\bigg)r^{d-1}dr.
			$$
		In particular, this energy decomposes into the following pieces
			$$
				E_2[\psi]=\frac{1}{2}\int_0^\infty\bigg(\big(\partial_t\psi\big)^2+\big(\partial_r\psi\big)^2+(d-1)\frac{\sin^2(\psi)}{r^2}\bigg)r^{d-1}dr
			$$
		and
			$$
				E_4[\psi]=\frac{d-1}{2}\int_0^\infty\sin^2(\psi)\bigg(\big(\partial_t\psi\big)^2+\big(\partial_r\psi\big)^2+\frac{(d-2)}{2}\frac{\sin^2(\psi)}{r^2}\bigg)r^{d-3}dr
			$$
		so that 
			$$
				E[\psi]=\alpha E_2[\psi]+\beta E_4[\psi].
			$$ 
		Given a solution $\psi$ of Equation \eqref{skyrme eqn}, denote by $\psi_\lambda$ its rescaling
			$$
				\psi_\lambda(t,r)=\psi(\lambda^{-1}t,\lambda^{-1}r),\quad\lambda>0.
			$$
		We then obtain the following relationship between the energy of $\psi$ and that of $\psi_\lambda$
			\begin{equation}
				E[\psi_\lambda]=\alpha\lambda^{d-2}E_2[\psi]+\beta\lambda^{d-4}E_4[\psi]. \label{energy rescale}
			\end{equation}
		If $d\geq5$, then $E[\psi_\lambda]\to0^+$ as $\lambda\to0^+$ assuming $E_2[\psi],E_4[\psi]<\infty$. Moreover, for $\lambda\ll1$, the leading-order term in \eqref{energy rescale} is $\beta\lambda^{d-4}E_4[\psi]$. Motivated by this observation, we momentarily restrict our attention to the special case of Equation \eqref{skyrme eqn} with $\alpha=0$ which reads
			\begin{equation}
				\frac{\sin^2(\psi)}{r^2}\Big(\partial_t^2\psi -\partial_r^2\psi-\frac{d-3}{r}\partial_r\psi\Big)+\frac{\sin(2\psi)}{2 r^2}\bigg(\big(\partial_t\psi\big)^2-\big(\partial_r\psi\big)^2+\frac{(d-2)\sin^2(\psi)}{r^2}\bigg)=0. \label{sf skyrme eqn}
			\end{equation}
		Note that solutions of Equation \eqref{sf skyrme eqn} conserve $E_4[\psi]$. Equation \eqref{sf skyrme eqn} is called the equation of motion of the \textit{co-rotational, strong field Skyrme model}.

		In \cite{CMS23}, Chen, Sch\"orkhuber, and the author introduced an explicit self-similar solution of Equation \eqref{sf skyrme eqn} which exists in any dimension $d\geq5$ and is given by 
			\begin{align}
				\psi^T(t,r)=U\Big(\frac{r}{T-t}\Big),\quad  T>0 \label{sf skyrme soln}
			\end{align}
		with the profile 
			\begin{equation}
 				U(\rho)=\arccos\left (\frac{a- b \rho^2}{a+\rho^2} \right ) \label{Profile_U_d}
			\end{equation}
		where $a:=\frac{1}{3} \left(2 (d-4)+\sqrt{3(d-4) (d-2)}\right)$ and $b :=2 \sqrt{\frac{d-4}{3(d-2)}}+1$. Observe that $U$ is smooth for $\rho \in [0, \rho^*]$, where $\rho^* = \sqrt{\frac{2a}{b-1}} > 1$ and $U(\rho^*) = \pi$. In particular, $\psi^T$ is smooth within the radial backward light cone
			$$
				\mc C_T:=\{(t,r):0 \leq t<T,0\leq r\leq T-t\}
			$$
		 and suffers a gradient blowup at the origin as $t \to T^{-}$ since 
			$$  
				|\partial_r\psi^T(t,0)|=\frac{c}{T-t}
			$$
		for some $c>0$ depending on $d$. Moreover, the solution of Equation \eqref{sf skyrme eqn} in $\mc C_1$ with the initial data 
			$$
				\big(\psi(0,r),\partial_t\psi(0,r)\big)=\big(U(r),\Lambda U(r)\big),
			$$
		where $\Lambda:=r\partial_r$, is exactly $\psi^1$.

	\subsection{The main result}\label{The Main Result}
		We restrict our attention to $d=5$. In this case, the expression for the blowup profile \eqref{Profile_U_d} can be equivalently written as 
			$$
				U(\rho)=2\arctan\Big(\frac{2\rho}{\sqrt{5-\r^2}}\Big).
			$$
		We exhibit stable blowup at the self-similar rate for Equation \eqref{skyrme eqn} within a backwards light cone and provide a precise description of the asymptotic blowup profile in terms of $U$. 
		
		\begin{theorem} \label{actual main result}
			Fix $d=5$ and $\alpha=\beta=1$. There exist positive constants $\delta,C,\omega$ such that the following holds.  For any $\lambda\in(0,\delta]$ and real-valued functions $(\psi_0,\psi_1)\in H^6(\mathbb B_{2\delta}^7)\times H^5(\mathbb B_{2\delta}^7)$ with
				$$
					\||\cdot|^{-1}(\psi_0,\psi_1)\|_{H^6(\mathbb B_{2\delta}^7)\times H^5(\mathbb B_{2\delta}^7)}\leq\lambda^{5/2}\frac{\delta}{C},
				$$
			there exists a unique $T\in[1-\delta,1+\delta]$ and a unique solution $\psi\in C^2(\mc C_{\lambda T})$ of Equation \eqref{skyrme eqn} depending Lipschitz continuously on $(\psi_0,\psi_1)$ and $\lambda$ such that $\psi(t,0)=0$ for $t\in[0,\lambda T)$ and 
				$$
					\psi(0,r)=U(\lambda^{-1}r)+\psi_0(r),\;\partial_t\psi(0,r)=\Lambda U(\lambda^{-1}r)+\psi_1(r).
				$$
			Moreover, the solution has the decomposition
				$$
					\psi(t,r)=U\bigg(\frac{r}{\lambda T-t}\bigg)+\phi\bigg(t,\frac{r}{\lambda T-t}\bigg)
				$$
			with 
				$$
					\||\cdot|^{-1}\phi(t,\cdot)\|_{H^5(\B^7)}+\big\||\cdot|^{-1}\big((\lambda T-t)\partial_t\phi(t,\cdot)+\Lambda\phi(t,\cdot)\big)\big\|_{H^4(\B^7)}\leq\delta  (T-t)^{\omega},
				$$
			for all $t \in [0,\lambda T)$.
		\end{theorem}
		
		\begin{remark}\label{On the time of blowup}[On the time of blowup]
			Consider the special case of Theorem \ref{actual main result} with $\alpha=0$ and $\psi_0,\psi_1=0$. From the previous discussion, we know that the solution within $\mathcal C_\lambda$ is $\psi^\lambda$. According to \cite{CMS23}, when $\alpha=0$ and $\psi_0,\psi_1$ are small in $H^6(\mathbb B_{2\delta}^7)\times H^5(\mathbb B_{2\delta}^7)$, there is a unique $T\in[1-\delta,1+\delta]$ and unique solution $\psi$ in $\mathcal C_{\lambda T}$ as in Theorem \ref{actual main result}. In this case, the shift of the blowup time is due to the fact that $(\psi^T)_{T>0}$ represents a one-parameter family of solutions which are related via time translations of each other. This symmetry yields an unstable eigenvalue of the corresponding linearization of Equation \eqref{sf skyrme eqn} around $\psi^T$ for any $T>0$. In \cite{CMS23}, the shift of blowup time suppresses this instability. However, for $\alpha\neq0$, this one-parameter family is no longer a family of solutions and so there is nothing obvious to linearize Equation \eqref{skyrme eqn} around. Regardless, the shift of blowup time is necessary given that we use the evolution generated by the linearization of Equation \eqref{sf skyrme eqn}, the strong field equation, around $\psi^{\lambda T}$ in order to control our solution of Equation \eqref{skyrme eqn}, the full equation. See Section \ref{Outline of the proof} and \ref{Operator and Semigroup Theory} for more details.
		\end{remark}
		
		\begin{remark}[One-parameter family of blowup solutions] \label{One-parameter family of blowup solutions}
				By taking $(\psi_0,\psi_1)=(0,0)$, Theorem \ref{actual main result} shows that the five-dimensional Skyrme model possesses a one-parameter family of solutions which blowup at the self-similar rate. This family of solutions is parameterized by $\lambda$ which can be at most a potentially very small number $\delta$. The solution corresponding to the parameter $\lambda$ blows up at time $\lambda T$ for a unique $T\in[1-\delta,1+\delta]$. In this sense, we can think of $\lambda$ as parameterizing the blowup times for this family of solutions. Since $\delta$ is potentially a very small number, the solutions we construct blowup very quickly. It remains unclear if this interval can be extended, i.e., if solutions with large blowup times exist.
		\end{remark}

		\begin{remark}[Regularity]
			The regularity assumptions in Theorem \ref{main result} and Corollary \ref{actual main result} ensure $L^{\infty}$-control of the perturbation and its time derivative. This allows us to define and control the seemingly degenerate nonlinearity in a sufficiently small neighborhood of $\psi^{\lambda T}$. In addition, we assume smallness of the initial data in an even stronger topology which we use to obtain Lipschitz-dependence of the blowup time via a fixed point argument and regularity of the solution up to the time of blowup.
		\end{remark}

		\subsubsection{Reformulation as a semilinear wave equation on $\R^{1+7}$} \label{Reformulation as a semilinear wave equation}
			Theorem \ref{actual main result} is proven by first reformulating it as a related semilinear wave equation. We then prove an auxiliary result, Theorem \ref{main result}, upon which Theorem \ref{actual main result} follows. The paper is then devoted to proving Theorem \ref{actual main result}.
			
			We begin this reformulation by rescaling the coordinates
				\begin{equation}
					(t,r)\mapsto(t',r')=\big(\lambda^{-1}t,\lambda^{-1}r\big),\quad\lambda>0
\label{rescaling transformation}.
				\end{equation}
			Then $\psi$ is a solution of Equation \eqref{skyrme eqn} if and only if $\tilde\psi$ is a solution of 
				\begin{align}
				\begin{split}
					\Big(\alpha\lambda^2&+\frac{4\beta\sin^2(\tilde \psi)}{(r')^2}\Big)\big(\partial_{t'}^2\tilde \psi-\partial_{r'}^2\tilde \psi\big)-\frac{4}{r'}\Big(\alpha\lambda^2+\frac{2\beta\sin^2(\tilde \psi)}{(r')^2}\Big)\partial_{r'}\tilde \psi
					\\
					&+\frac{4\sin(2\tilde \psi)}{2(r')^2}\bigg(\alpha\lambda^2+\beta\Big(\big(\partial_{t'}\tilde \psi\big)^2-\big(\partial_{r'}\tilde \psi\big)^2+\frac{3\sin^2(\tilde \psi)}{(r')^2}\Big)\bigg)=0 \label{rescaled skyrme eqn}
				\end{split}
				\end{align}
			where $\psi(t,r)=\tilde\psi(\lambda^{-1}t,\lambda^{-1}r)$. With this observation made, we abuse notation and write $(t,r)$ instead of $(t',r')$ and $\psi$ instead of $\tilde \psi$ and clarify the identification when necessary. Separating the wave maps and strong field terms in Equation \eqref{rescaled skyrme eqn} yields
				\begin{align*}
					\alpha\lambda^2&\bigg(\partial_t^2\psi-\partial_r^2\psi-\frac{4}{r}\partial_r\psi+\frac{2\sin(2\psi)}{r^2}\bigg)
					\\
					&+\frac{4\beta\sin^2(\psi)}{r^2}\bigg(\partial_t^2\psi-\partial_r^2\psi-\frac{2}{r}\partial_r\psi+\cot(\psi)\Big(\big(\partial_t\psi\big)^2-\big(\partial_r\psi\big)^2\Big)+\frac{3\sin(2\psi)}{2r^2}\bigg)=0.
				\end{align*}
			The precise values of $\alpha$ and $\beta$ no longer play an imporant role and so we set $\alpha=\beta=1$.
			
			To ensure finite energy near the origin, we impose the boundary condition $\psi(t,0)=0$ for all $t$ and switch to the new independent variable 
				\begin{equation}
					u(t,r):=r^{-1}\psi(t,r). \label{rescaled soln}
				\end{equation}
			The self-similar solution \eqref{sf skyrme soln} of the strong field Skyrme model transforms into
				$$
					u^T(t,r):=r^{-1}\psi^T(t,r) = \frac{1}{T-t} \tilde U\left (\frac{r}{T-t} \right ),
				$$
			for $\tilde U(\rho) := \rho^{-1} U(\rho)$. Equation \eqref{rescaled skyrme eqn} becomes
				\begin{align*}
					&\lambda^2\bigg(\partial_t^2u-\partial_r^2u-\frac{6}{r}\partial_ru-f_{WM}(ru,r)\bigg)
					\\
					&+\frac{4\sin^2(ru)}{r^2}\bigg(\partial_t^2u-\partial_r^2u-\frac{6}{r}\partial_ru-f_{SF}(ru,r\partial_ru,r\partial_tu,r)\bigg)=0
				\end{align*}
			where
				$$
					f_{WM}(ru,r)=\frac{4ru-2\sin(2ru)}{r^3}
				$$
			and
				\begin{align}
				\begin{split}
					f_{SF}\big(r u,r\partial_r u,r\partial_t u,r\big)=&-\frac{1}{r}\cot(r u)\big((r\partial_t u)^2-(r\partial_r u)^2\big)-\frac{2}{r^2}\big(1-r u\cot(r u)\big)r\partial_r u
					\\
					&-\frac{\frac{3}{2}\sin(2r u)-2r u-(r u)^2\cot(r u)}{r^3} \label{def:fsf}
				\end{split}
				\end{align}	
			denote the wave maps and strong field Skyrme model nonlinearities respectively. By adding and subtracting $f_{SF}$ in the first line and regrouping terms, we obtain
				\begin{align*}
					\bigg(\lambda^2+&\frac{4\sin^2(ru)}{r^2}\bigg)\bigg(\partial_t^2u-\partial_r^2u-\frac{6}{r}\partial_ru-f_{SF}(ru,r\partial_ru,r\partial_tu,r)\bigg)
					\\
					&+\lambda^2\big(f_{SF}(ru,r\partial_ru,r\partial_tu,r)-f_{WM}(ru,r)\big)=0.
				\end{align*}
			Dividing through by the coefficient on the first term, we obtain the semilinear wave equation
				\begin{equation}
					\partial_t^2u-\partial_r^2u-\frac{6}{r}\partial_ru-f_{SF}(ru,r\partial_ru,r\partial_tu,r)-\lambda^2G_\lambda\big(r u,r\partial_r u,r\partial_t u,r\big)=0 \label{semilinear skyrme eqn}
				\end{equation}
			where
				$$
					G_\lambda\big(r u,r\partial_r u,r\partial_t u,r\big)=g_\lambda\big(ru,r\big)G\big(r u,r\partial_r u,r\partial_t u,r\big)
				$$
			with
				$$
					g_\lambda\big(ru,r\big)=\bigg(\lambda^2+\frac{4\sin^2(ru)}{r^2}\bigg)^{-1}
				$$
			and
				$$
					G\big(r u,r\partial_r u,r\partial_t u,r\big)=f_{WM}(ru,r)-f_{SF}(ru,r\partial_ru,r\partial_tu,r).
				$$
			In this sense, the original quasilinear wave equation, Equation \eqref{skyrme eqn}, is effectively a $(1+7)$-dimensional semilinear wave equation.
			
			In the following, we denote the (non-radial) backward light cone in $(1+7)$-dimensions by 
				$$ 
					\mathfrak C_T := \{(t,x) \in [0,T) \times \R^7: 	 |x| \leq T-t \}.
				$$
			For notation and conventions, see Section \ref{Notation and conventions}.
				\begin{theorem}\label{main result}
					There exist positive constants $\delta,C,\omega$ such that the following holds. For any $\lambda\in(0,\d]$ and $(f,g)\in H^6(\mathbb B_{2}^7)\times H^5(\mathbb B_{2}^7)$ be real-valued functions with
						$$
							\|(f,g)\|_{H^6(\mathbb B_{2}^7)\times H^5(\mathbb B_{2}^7)}\leq\frac{\delta}{C},
						$$
					there exists a unique $T\in[1-\delta,1+\delta]$ and a unique solution $u\in C^2_\text{rad}(\mathfrak C_T)$ of Equation \eqref{semilinear skyrme eqn} depending Lipschitz continuously on $(f,g)$ and $\lambda$ such that
						\begin{equation}
							u(0,r)=\tilde U(r)+f(r),\;\partial_tu(0,r)=(1+\Lambda)\tilde U(r)+g(r). \label{initial data}
						\end{equation}
					Moreover, the solution has the decomposition
						\begin{equation}
							u(t,r)  = \frac{1}{T-t} \left [\tilde U \bigg (\frac{r}{T-t} \bigg ) + \varphi \bigg (t,\frac{r}{T-t}  \bigg) \right ] \label{rescaled decomposition}
						\end{equation}
					with 
						\begin{equation} 
							\|\varphi(t,\cdot)\|_{H^5(\B^7)}+\big\|(T-t)\partial_t\varphi(t,\cdot)+(1+\Lambda)\varphi(t,\cdot)\big\|_{H^4(\B^7)} \leq\delta  (T-t)^{\omega},\label{convergence}
						\end{equation}
					for all $t \in [0,T)$. 
				\end{theorem}
			By undoing the scaling transformation \eqref{rescaling transformation}, we obtain stable blowup for Equation \eqref{skyrme eqn} with $\alpha=\beta=1$ and $d=5$.
				
				\begin{proof}[Proof of Theorem \ref{main result} $\implies$ Theorem \ref{actual main result}]
					Set 
						$$
							f(r)=r^{-1}\psi_0(\lambda r),\quad g(r)=\lambda r^{-1}\psi_1(\lambda r),\quad r\in(0,2].
						$$
					Note that this pointwise definition is well-defined via Sobolev embedding. By rescaling in Sobolev spaces,
						$$
							\|(f,g)\|_{H^6(\mathbb B_{2}^7)\times H^5(\mathbb B_{2}^7)}\leq\lambda^{-5/2}\||\cdot|^{-1}(\psi_0,\psi_1)\|_{H^6(\mathbb B_{2\delta}^7)\times H^5(\mathbb B_{2\delta}^7)}\leq\frac{\d}{C}.
						$$
					As a consequence, $(f,g)\in H^6(\mathbb B_{2}^7)\times H^5(\mathbb B_{2}^7)$ satisfies the hypotheses of Theorem \ref{main result}. Thus, there exists a unique $T\in[1-\d,1+\d]$ and a unique $u\in C^2_\text{rad}(\mathfrak C_T)$ solving Equation \eqref{semilinear skyrme eqn} with initial data as in \eqref{initial data} and satisfying the decomposition \eqref{rescaled decomposition}. According to the calculations carried out in Section \ref{Reformulation as a semilinear wave equation}, we have that
						$$
							\psi(t,r)=\lambda^{-1}ru(\lambda^{-1}t,\lambda^{-1}r),
						$$
					is a solution of Equation \eqref{skyrme eqn} in $\mathcal C_{\lambda T}$. In particular, for $r\in(0,2\delta]$,
						$$
							\psi(0,r)=\lambda^{-1}ru(0,\lambda^{-1}r)=\lambda^{-1}r\tilde U(\lambda^{-1}r)+\lambda^{-1}rf(\lambda^{-1}r)=U(\lambda^{-1}r)+\psi_0(r)
						$$
					and
						$$
							\partial_t\psi(0,r)=\lambda^{-1}r\partial_0u(0,\lambda^{-1}r)=\lambda^{-1}\Big(\lambda^{-1}r(1+\Lambda)\tilde U(\lambda^{-1}r)+\lambda^{-1}rg(\lambda^{-1}r)\Big)=\lambda^{-1}\Lambda U(\lambda^{-1}r)+\psi_1(r)
						$$
					along with the decomposition
						\begin{align*}
							\psi(t,r)&=\frac{\lambda^{-1}r}{T-\lambda^{-1}t} \left [\tilde U \bigg (\frac{\lambda^{-1}r}{T-\lambda^{-1}t} \bigg ) + \varphi \bigg (\lambda^{-1}t,\frac{\lambda^{-1}r}{T-\lambda^{-1}t}  \bigg) \right ]
							\\
							&=U \bigg (\frac{r}{\lambda T-t} \bigg ) +\phi\bigg (t,\frac{r}{\lambda T-t}  \bigg)
						\end{align*}
					where $\phi(t,\rho)=\rho\varphi(\lambda^{-1}t,\rho)$. Lipschitz dependence and decay follow from that of $u$.
				\end{proof}
		
	\subsection{Discussion of related results and history of the Skyrme model}\label{Discussion of related results and history of the Skyrme model}
		We briefly review some results concerning the global dynamics of the Skyrme model. 
		
		\subsubsection{Results outside of co-rotational symmetry}
			In \cite{W11}, Wong provided a general framework for studying the hyperbolicity, and thus local well-posedness, of Lagrangian field theories. In particular, Wong proved local well-posedness for the dynamical problem of the 3D Skyrme model outside of co-rotational symmetry for almost stationary initial data. On the other hand, Wong also showed that the Skyrme model can suffer a breakdown of hyperbolicity, and thus ill-posedness of the Cauchy problem, corresponding to a change of signature of its principal symbol nearby certain initial data sets. No results for the Skyrme model outside of co-rotational symmetry in other dimensions appear to exist.
		
		\subsubsection{Results inside of co-rotational symmetry for the 3D Skyrme model}
			For small data in critical Besov-type spaces, global existence and scattering was first proven by Geba, Nakanishi, and Rajeev \cite{GNR11}. Soon after\footnote{Though published in 2021, this result first appeared as a preprint in 2012.}, Li \cite{L21} proved global regularity for arbitrary, large radial initial data in $H^k(\mathbb R^5)\times H^{k-1}(\mathbb R^5)$ for integer $k\geq4$. Li's approach to large data global regularity was then extended by Geba and Grillakis \cite{GG18} to data $H^s(\mathbb R^5)\times H^{s-1}(\mathbb R^5)$ for non-integer $s>\frac{7}{2}$. 
			
			The existence of the Skyrmion was proven by Kapitanskii and Ladyzhenskaya \cite{KL83} via variational techniques. Another proof was obtained by Mcleod and Troy \cite{MT91} using ODE methods. For a comprehensive overview of topological solitons, we refer the reader to the book by Manton and Sutcliffe \cite{MS04}. The nonlinear asymptotic stability of the Skyrmion was investigated numerically by Bizo\'n, Chmaj, and Rostworowski \cite{BCR07} and a rigorous proof of its linear stability was proved by Creek, Donninger, Schlag, and Snelson \cite{CDSS16}. A proof of the Skyrmion's nonlinear asymptotic stability remains open.
			
		\subsubsection{Co-rotational results for other dimensions}
			Geba and Da Silva \cite{GDs13} showed that energy does not concentrate for the $d=2$ Skyrme model. Soon after, Creek \cite{C13} proved large data global regularity for Equation \eqref{skyrme eqn} with $d=2$ is globally well-posed for large data in $H^s(\mathbb R^5)\times H^{s-1}(\mathbb R^5)$ for $s\geq4$. Around the same time, Geba, Nakanishi, and Zhang \cite{GNZ15} showed that small data in critical Besov spaces evolve into global solutions and scatter. Creek's result was then extended to large data in $H^s(\mathbb R^5)\times H^{s-1}(\mathbb R^5)$ for $s>3$ by Geba and Grillakis \cite{GG17}.
			
			Despite the large data global regularity results in dimensions $2$ and $3$, the present work shows that there is a transition to the breakdown of regularity for $d=5$ even within co-rotational symmetry. It is currently unclear if this happens earlier with $d=4$. Moreover, the present work shows that Skyrme's correction term provides a mechanism by which singularities can form for the full model. This is due to the fact that the strong field limit of the Skyrme model already exhibits stable self-similar blowup for $d=5$ and, in a precise sense used presently, the wave maps portion of the equation is negligible nearby this self-similar solution. The existence of self-similar solutions was first proven by the author in \cite{M20} using variational arguments al\'a Shatah \cite{S88}. Chen, Sch\"orkhuber, and the author \cite{CMS23} found what is believed to be this solution in closed form, namely \eqref{sf skyrme soln}. Moreover, \cite{CMS23} shows that this family of solutions is asymptotically stable within backwards light cones.
			
	\subsection{Outline of the proof}\label{Outline of the proof}
		Our argument is perturbative in the sense that we seek a solution of Equation \eqref{semilinear skyrme eqn} of the form $u=u^T+v$ for some $T>0$ and $v$ and then analyze the equation that $v$ must solve. Since $u^T$ solves 
			$$
				\partial_t^2u^T-\partial_r^2u^T-\frac{6}{r}\partial_ru^T-f_{SF}\big(ru^T,r\partial_ru^T,r\partial_tu^T,r\big)=0,
			$$
		we obtain an equation for $v$ of the form
			\begin{align}
			\begin{split}
				\partial_t^2v-&\partial_r^2v-\frac{6}{r}\partial_rv-\tilde f_{SF}\big(rv,r\partial_rv,r\partial_tv,r\big)
				\\
				&-\lambda^2G_\lambda\big(r(u^T+v),r\partial_r(u^T+v),r\partial_t(u^T+v),r\big)=0 \label{expanded rescaled skyrme eqn}
			\end{split}
			\end{align}
		where $\tilde f_{SF}$ denotes the linear and nonlinear terms in the expansion of $f_{SF}$ around $u^T$. However, since $u^T$ is not an exact solution of Equation \eqref{semilinear skyrme eqn}, we obtain zeroth-order and linear (in $v$) terms in the second line of \eqref{expanded rescaled skyrme eqn}. More precisely, Taylor expansion yields
			\begin{align*}
				G_\lambda&\big(r(u^T+v),r\partial_r(u^T+v),r\partial_t(u^T+v),r\big)
				\\
				&=\underbrace{G_\lambda\big(ru^T,r\partial_ru^T,r\partial_tu^T,r\big)}_{\text{zeroth-order}}+\underbrace{DG_\lambda\big(ru^T,r\partial_ru^T,r\partial_tu^T,r\big)\cdot\big(rv,r\partial_rv,r\partial_tv\big)}_{\text{linear}}+\;h.o.t..
			\end{align*}
		Without additional smallness of these terms, a perturbative approach is doomed to fail.
		
		This issue can be remedied \textit{precisely by the lack of scale-invariance} of the Skyrme model. Namely, the smallness of the scale $\lambda$ paired with the precise difference in scaling of the wave maps and strong field terms forces smallness of the zeroth-order and linear terms up to the time of blowup. The latter feature can be made easily visible by transforming to similarity coordinates defined by the equations
			$$ 
				\tau=\log\Big(\frac{T}{T-t}\Big),\quad\rho=\frac{r}{T-t}.
			$$
		Upon transforming Equation \eqref{expanded rescaled skyrme eqn}, we see that the problematic term has the following schematic expansion
			$$
				\lambda^2G_\lambda\big(r(u^T+v),r\partial_r(u^T+v),r\partial_t(u^T+v),r\big)|_{(t,r)=\big(t(\tau,\rho),r(\tau,\rho)\big)}\sim\lambda^2T^2e^{-2\t}(1+v+\partial v+h.o.t.).
			$$
		Here, $T$ should be thought of as size $1$. Thus, the initial size of the zeroth-order and linear terms can be made small so long as $\lambda\ll1$ with their size at later times becoming even smaller due to the decaying exponential factor. This decaying exponential is present precisely due to the difference in scaling of the wave maps and strong field terms; the wave maps terms scale like $\lambda^{-2}$ while the strong field terms scale like $\lambda^{-4}$ upon rescaling coordinates as in \eqref{rescaling transformation}. 
		
		These observations suggest the following strategy. We evolve initial data of the form
			$$
				u(0,r)=\tilde U(r)+f(r),\quad\partial_tu(0,r)=(1+\Lambda)\tilde U(r)+g(r)
			$$
		according to Equation \eqref{semilinear skyrme eqn} for $(f,g)$ sufficiently small and $\lambda>0$ sufficiently small and seek a solution of the form $u=u^T+v$ for some $T$ to be determined. Note that the initial data corresponds to a small perturbation of the data $\big(u^1(0,r),\partial_tu^1(0,r)\big)$. We then show that $T$ can be chosen uniquely and close to $1$ and that $v$ can be uniquely solved for and, in a precise sense, blows up slower than $u^T$ within $\mathcal C_T$. 
		
		We achieve this by using the standard approach of transforming to a first-order system in similarity coordinates. Namely, the evolution of perturbations of $u^T$ is governed by an operator equation of the form
				\begin{equation}
					\partial_\t\Phi(\tau)=(\mathbf L_0+\mathbf L')\Phi(\tau)+\mathbf N(\Phi(\tau))+\mathbf{G}_{\lambda,T}\big(\tau,\Phi(\tau)\big) \label{outline: operator skyrme eqn}
				\end{equation}
			for $\Phi(\tau) = (v_1,v_2)$ where $v_1$ and $v_2$ are suitable rescalings of $v$ and $\partial_tv$ expressed in similarity coordinates. The equation obtained by dropping $\mathbf{G}_{\lambda,T}$, 
				$$
					\partial_\t \Phi(\tau)= (\mathbf L_0+\mathbf L')\Phi(\tau) +\mathbf N(\Phi(\tau)),
				$$
			is exactly the strong field equation. Moreover, $\mathbf L_0=\mathbf L_W+\mb L_D$, where $\mb L_W$ is the generator of free waves in similarity coordinates and $\mb L_D$ translates into a scale-invariant damping term in physical coordinates. The term $\mathbf L'$ is a bounded, relatively compact perturbation of $\mathbf L_0$ and derives from the linear terms obtained by expanding $f_{SF}$ around $u^T$. The operator $\mb N$ is the remaining nonlinearity obtained by expanding $f_{SF}$ around $u^T$.
			
			To solve Equation \eqref{outline: operator skyrme eqn}, we recall the linear theory developed in \cite{CMS23}. In particular, it states that $\mb L = \mb L_0 + \mb L'$ is the generator of a strongly continuous semigroup of bounded operators $\big(\mathbf S(\tau)\big)_{\tau\geq0}$. We then seek the existence of a unique, global strong solution of Equation \eqref{outline: operator skyrme eqn} which decays exponentially. That is, we seek a unique global and exponentially decaying solution of the integral equation
				\begin{equation}
					\Phi(\tau)=\mathbf S(\tau)\Phi(0)+\int_0^\tau\mathbf S(\tau-s)\Big(\mathbf N\big(\Phi(s)\big)+\mathbf{G}_{\lambda,T}\big(s,\Phi(s)\big)\Big)ds. \label{fpp}
				\end{equation}
			
			The success of this approach requires sufficient decay of the semigroup. We recall from \cite{CMS23} that the spectrum of $\mathbf L_0$ is contained in the set $\{\lambda\in\mathbb C:\Re\lambda\leq-\tfrac{1}{2}\}$. Combining the boundedness and relative compactness of $\mathbf L'$, the analytic Fredholm theorem guarantees that for any $\varepsilon>0$, the spectrum of $\mathbf L$ possesses a discrete set of eigenvalues, all of which have finite algebraic multiplicity, in the set $\{\lambda\in\mathbb C:\Re\lambda\geq-\tfrac{1}{2}+\varepsilon\}$. Note that, in contrast, a compact perturbation only introduces a finite set of such eigenvalues in that set. Thus, there may be infinitely many unstable eigenvalues, i.e. those with $\Re\lambda\geq0$, obstructing decay of the semigroup. 
			
			We utilize the detailed spectral analysis of $\mathbf L$ from \cite{CMS23} to circumvent this issue. In particular, it is shown that $1$ is an eigenvalue with unit algebraic and geometric multiplicity. The existence of this unstable eigenvalue is guaranteed by the time translation symmetry of Equation \eqref{sf skyrme eqn} as discussed in Remark \ref{On the time of blowup}. Moreover, it is also shown that any other spectrum is contained in $\{\lambda\in\mathbb C:\Re\lambda<0\}$. This is highly nontrivial and requires the resolution of the corresponding `mode stability problem' (see \cite{D23} for a detailed history of these problems). 
			
			This is still insufficient to obtain decay of the semigroup away from the eigenspace of $1$ since there may be infinitely many eigenvalues in the strip $\{\lambda\in\mathbb C:-\tfrac{1}{2}+\varepsilon\leq\Re\lambda<0\}$. Remarkably, \cite{CMS23} shows that there are only \textit{finitely many eigenvalues} in this strip by using the precise structure of $\mathbf L'$. This is achieved by using the fact that the problematic terms in $\mathbf L'$ causing it to be compact \textit{relative} to $\mathbf L_0$, and not just simply \textit{compact}, are derived from the nonlinear term $\cot(\psi)\big((\partial_t\psi)^2-(\partial_r\psi)^2\big)$. As a consequence, \cite{CMS23} shows that $\mathbf L_0+\mathbf L'$ conjugates to a compact perturbation of $\mathbf L_0$ whose spectrum can be characterized sufficiently. More precisely, \cite{CMS23} shows that there exists a bounded invertible operator $\mathbf\Gamma$ and compact operator $\mathbf V$ such that 
				$$
					\mathbf\Gamma(\mathbf L_0+\mathbf L')\mathbf\Gamma^{-1}=\mathbf L_0+\mathbf V
				$$
			and that the spectrum of $\mathbf L_0+\mathbf V$ is contained in $\{\lambda\in\mathbb C:\Re\lambda\leq-\tfrac{1}{2}+\varepsilon\}\cup S_\varepsilon\cup\{1\}$ with $\varepsilon<\frac{1}{2}$ and $S_\varepsilon$ being a finite set of eigenvalues whose real parts are all negative. As a consequence, decay of the semigroup away from the eigenspace of $1$ can be obtained via resolvent estimates. For precise details, we refer the reader to Section \ref{Operator and Semigroup Theory} and Section 3 of \cite{CMS23}. 
			
			A second and necessary ingredient for carrying out our approach to solving Equation \eqref{outline: operator skyrme eqn} is a local Lipschitz estimate on $\mathbf N$ and a uniform (in $\tau$) Lipschitz estimate on $\mathbf{G}_{\lambda,T}$. The former is proved in \cite{CMS23} and the latter is proved in Section \ref{Existence of strong solutions}. These estimates allow us to solve Equation \eqref{fpp} as a fixed-point problem on an appropriately chosen Hilbert space. In particular, $\mathbf{G}_{\lambda,T}$ is not locally Lipschitz precisely due to the presence of linear terms. As described previously, this is remedied by the smallness of $\lambda$. After noting this, we are able to employ a contraction mapping argument as in \cite{CMS23} to obtain the desired decaying solution. The main result is then proved in Section \ref{Proof of the main result} and is obtained by undoing the transformation to a first-order system in similarity coordinates.
				
\subsection{Notation and conventions}\label{Notation and conventions}
Given $R>0$ and $n\in\mathbb N$, we denote by $\mathbb B_R^n:=\{x\in\mathbb R^n:|x|<R\}$ the open ball in $\mathbb R^n$ of radius $R$ centered at the origin. When $R=1$, we drop the subscript and simply write $\mathbb B^n$. By $\mathbb H$ we denote the open right-half plane in $\mathbb C$, i.e., $\mathbb H:=\{z\in\mathbb C:\Re z>0\}$. On a Hilbert space $\mathcal H$, we denote by $\mathcal B(\mathcal H)$ the space of bounded linear operators. For a closed operator $L$ on a Hilbert space $\mathcal H$ with domain $\mathcal D(L)$, we denote its resolvent set by $\rho(L)$ and by $R_L(\lambda):=(\lambda I- L)^{-1}$ the resolvent operator for $\lambda\in\rho( L)$. Furthermore, we denote by $\sigma( L):=\mathbb C\setminus\rho( L)$ the spectrum of $L$ and by $\sigma_p( L)$ its point spectrum. As we will only work with strongly continuous semigroups $\big(S(\tau)\big)_{\tau\geq0}$ of bounded operators on $\mathcal H$, we will instead refer to these more simply as semigroups on $\mathcal H$ whenever necessary. Given $x,y\geq0$, we say $x\lesssim y$ if there exists a constant $C>0$ such that $x\leq Cy$. Furthermore, we say that $x\simeq y$ if $x\lesssim y$ and $y\lesssim x$. If the constant $C$ depends on a parameter, say $k$, we will write $x\lesssim_ky$ when it is important to note the dependence on this parameter. 

\subsubsection{Function spaces}\label{Function Spaces}
For $R>0$, let
\[ C^\infty_\text{rad}(\overline{\mathbb B_R^7})=\{u\in C^\infty(\overline{\mathbb B_R^7}):u\text{ is radial}\}. \]
For $k\in\N$, we define the radial Sobolev space $H_\text{rad}^k(\mathbb B_R^7)$ as the completion of $C^\infty_\text{rad}(\overline{\mathbb B_R^7})$ under the standard Sobolev norm
\[ \|u\|_{H^k(\mathbb B_R^7)}^2:=\sum_{|\alpha|\leq k}\|\partial^\alpha u\|_{L^2(\mathbb B_R^7)}^2 \]
with $\alpha\in\mathbb N_0^7$ denoting a multi-index with $\partial^\alpha u=\partial_1^{\alpha_1}\dots\partial_d^{\alpha_d}u$ and $\partial_iu(x)=\partial_{x^i}u(x)$. Similarly, we denote by $H^k(\mathbb R^7)$ the completion of the Schwartz space $\mathcal S(\mathbb R^7)$. Furthermore, we define
			$$
				\mc H^k := H_{\text{rad}}^k(\mathbb B^7)\times H_{\text{rad}}^{k-1}(\mathbb B^7)
			$$
	which comes equipped with the norm 
			$$
				\|\mathbf u\|_{\mathcal H^k}^2:=\|u_1\|_{H^k(\mathbb B^7)}^2+\|u_2\|_{H^{{k-1}}(\mathbb B^7)}^2,\quad\mathbf u=(u_1,u_2)
			$$
	and the dense subset $C^\infty_\text{rad}(\overline{\mathbb B^7}) \times C^\infty_\text{rad}(\overline{\mathbb B^7})$. Central to our analysis is the space $\mc H^5$ which we will more simply denote as $\mc H$.

In many places it will be convenient to work with radial representatives of functions in $C^\infty_\text{rad}(\overline{\mathbb B_R^7})$. That is, for any function $u\in C^\infty_\text{rad}(\overline{\mathbb B_R^7})$, there is a function $\hat u:[0,R]\to\mathbb C$ such that $u(x)=\hat u(|x|)$ for all $x\in\overline{\mathbb B^7}$. In fact, by Lemma 2.1 of \cite{G22a}, we have $\hat u\in C_e^\infty[0,R]$ where $C_e^\infty[0,R]$ denotes the set of functions
\[
				C_e^\infty[0,R]:=\{u\in C^\infty[0,R]:u^{(2k+1)}(0)=0,\;k\in\N_0\}. \]	
		
\section{First-order formulation}\label{First-order formulation}
	In this section, we convert the equation
		\begin{equation}
			\partial_t^2u-\partial_r^2u-\frac{6}{r}\partial_ru-f_{SF}(ru,r\partial_ru,r\partial_tu,r)-\lambda^2G_\lambda\big(r u,r\partial_r u,r\partial_t u,r\big)=0 \label{main equation}
		\end{equation}
	with initial data
		\begin{equation}
			u(0,r)=\tilde U(r)+f(r),\;\partial_tu(0,r)=(1+\Lambda)\tilde U(r)+g(r), \label{initial data_}
		\end{equation}
	into a suitable first-order system in similarity coordinates. 

	Similarity coordinates are defined via the equation
		$$
			(\t,\r):=\bigg(\log\Big(\tfrac{T}{T-t}\Big),\frac{r}{T-t}\bigg),\quad T>0.
		$$
	Restricting to the backward light cone $\mc C_T$ implies that $\rho \in [0,1]$ and $\tau \in [0,\infty)$.
We furthermore introduce rescaled dependent variables, $\psi_1$ and $\psi_2$, as follows
		\begin{equation}
			\psi_1(\t,\r) =(T-t)u(t,r)|_{(t,r)=(t(\t,\r),r(\t,\r))},\;\psi_2(\t,\r)=(T-t)^2\partial_tu(t,r)|_{(t,r)=(t(\t,\r),r(\t,\r))}. \label{rescaling}
		\end{equation}
	Observe that this rescaling is natural from the point of view that the anticipated leading-order term, namely $u^T(t,r)$, is blowing up with rate $(T-t)^{-1}$. In this sense, we are tracking the evolution of the data \eqref{initial data_} relative to the blowup rate of $u^T$. 
	
	For a solution $u$ of Equation \eqref{main equation}, a direct calculation shows that $\psi_1$ and $\psi_2$ as defined in \eqref{rescaling} solve
		\begin{align*}
		\begin{pmatrix}
			\partial_\tau\psi_1
			\\
			\partial_\tau\psi_2
		\end{pmatrix}
		=&
		\begin{pmatrix}
			-\r\partial_\r-1&1
			\\
			\Delta_\text{rad}&-\r\partial_\r-2
		\end{pmatrix}
		\begin{pmatrix}
			\psi_1
			\\
			\psi_2
		\end{pmatrix}
		+
		\begin{pmatrix}
			0
			\\
			f_{SF}\big(\r\psi_1,\r\partial_\r\psi_1,\r\psi_2,\r\big)
		\end{pmatrix}
		\\
		&+
		\begin{pmatrix}
			0
			\\
			\lambda^2G_\lambda(\r\psi_1,\r\partial_\r\psi_1,\r\psi_2,\rho)
		\end{pmatrix}
		\end{align*}
	where $\Delta_\text{rad}=\partial_\r^2+\frac{6}{\r}\partial_\r$ denotes the seven-dimensional radial Laplacian. The self-similar solution of the strong field equation transforms as 
		$$
			\begin{pmatrix}
				(T-t)u^T(t,r)
				\\
				(T-t)^2 \partial_t u^T(t,r)
			\end{pmatrix}\Bigg|_{(t=t(\tau,\rho),r=r(\tau,\rho))}
			=
			\begin{pmatrix}
				\tilde U(\r)
				\\
				(1+\Lambda)\tilde U(\r)
			\end{pmatrix}
			=:
			\begin{pmatrix} 
				U_1(\rho) 
				\\ 
				U_2(\rho) 
			\end{pmatrix}.
		$$
		
	Inserting the ansatz 
		$$		
			\begin{pmatrix}
				\psi_1(\tau,\rho)
				\\
				\psi_2(\tau,\rho)
			\end{pmatrix} =	
			\begin{pmatrix}
				U_1(\rho)
				\\
				U_2(\rho)
			\end{pmatrix} 	
			+ 
			\begin{pmatrix}
				\varphi_1(\tau,\rho)
				\\
				\varphi_2(\tau,\rho)
			\end{pmatrix}
		$$ 
	yields 
		\begin{align}
		\begin{split}
			\begin{pmatrix}
				\partial_\tau\varphi_1
				\\
				\partial_\tau\varphi_2
			\end{pmatrix}
			=&
			\begin{pmatrix}
				-\r\partial_\r-1&1
				\\
				\Delta_\text{rad}&-\r\partial_\r-4
			\end{pmatrix}
			\begin{pmatrix}
				\varphi_1
				\\
				\varphi_2
			\end{pmatrix}
			+
			\begin{pmatrix}
				0
				\\
				V_1(\r)\varphi_1+ V_2(\rho)\rho\big(\rho \varphi_2-\partial_\r\varphi_1\big)
			\end{pmatrix}
			\\
			&+
			\begin{pmatrix}
				0
				\\
				N(\r\varphi_1,\r\partial_\r\varphi_1,\r\varphi_2,\r)
			\end{pmatrix} 
			+
			\begin{pmatrix}
				0
				\\
				\lambda^2T^2e^{-2\tau}\mathcal G_{\lambda,T}(\r\varphi_1,\r\partial_\r\varphi_1,\r\varphi_2,\t,\r)
			\end{pmatrix}\label{first order eqn}
		\end{split}
		\end{align}
	where
			\begin{align}
				V_1(\r) & =\r\partial_2 f_{SF}(\r,\r U_1,\r \partial_\rho U_1 ,\r U_2)=-\frac{5\big(21\rho^6-375\rho^4+1455\rho^2-2125\big)}{\big(5+3\rho^2\big)^2\big(5-\rho^2\big)^2}, \label{Def_V1}
			\end{align}
			\begin{align}
 				V_2(\rho):=-\frac{2\big(3\rho^2-35\big)}{\big(5+3\rho^2\big)\big(5-\rho^2\big)}, \label{V2_Def}
			\end{align}
			\begin{align*}
				N(\r\varphi_1,\r\partial_\r\varphi_1,\r\varphi_2,\r)  =& 
f_{SF}(\r,\r U_1+\r\varphi_1,\r \partial_\rho U_1  +\r\partial_\r\varphi_1,\r U_2+\r\varphi_2)\\
& -f_{SF}(\r,\r U_1,\r \partial_\rho U_1 ,\r U_2) -V_1(\rho) \varphi_1+\rho V_2(\r) \partial_\r\varphi_1 -\big(\rho^2 V_2(\r)-2\big)\varphi_2, 
			\end{align*}
		and
			$$
				\mathcal G_{\lambda,T}(\r\varphi_1,\r\partial_\r\varphi_1,\r\varphi_2,\tau,\r)=g_{\lambda,T}(\r U_1+\r\varphi_1,\tau,\rho)G(\r U_1+\r\varphi_1,\r U_1'+\r\partial_\r\varphi_1,\r U_2+\r\varphi_2,\r)
			$$
		with
			$$
				g_{\lambda,T}(\rho\psi_1,\tau,\rho)=\bigg(\lambda^2T^2e^{-2\t}+\frac{4\sin^2(\r\psi_1)}{\r^2}\bigg)^{-1}. 
			$$
		Regarding the precise form of the linear portion of Equation \eqref{first order eqn}, we refer the reader to Section \ref{Operator and Semigroup Theory} and to \cite{CMS23}. Furthermore, the initial data becomes
			\begin{align}\label{Initial_Data_Trans}
				\begin{pmatrix}
					\varphi_1(0,\rho)
					\\
		  			\varphi_2(0,\rho)
				\end{pmatrix} 
				= 
				\begin{pmatrix}
					Tu_1(T\rho)+TU_1(T\rho)-U_1(\rho)
					\\
	   				T^2u_2(T\rho)+T^2U_2(T\rho)-U_2(\rho)
				\end{pmatrix} 
			\end{align}
		In what follows, the main object of study is Equation \eqref{first order eqn} along with the initial data \eqref{Initial_Data_Trans}.
	
	\subsection{Formulation as an Abstract Initial Value Problem}
		In this section, we reformulate Equations \eqref{first order eqn} and \eqref{Initial_Data_Trans} as an abstract initial value problem on $\mathcal H$. We first define several linear and nonlinear operators on $\mathcal H$. Using the semigroup theory developed in \cite{CMS23}, we are able to obtain a suitable weak formulation of Equations \eqref{first order eqn} and \eqref{Initial_Data_Trans}.
		
		\subsubsection{Review of operator and semigroup theory from \cite{CMS23}}\label{Operator and Semigroup Theory}
			Let $\mathbf u=(u_1,u_2)\in C^\infty_\text{rad}(\overline{\mathbb B^7}) \times C^\infty_\text{rad}(\overline{\mathbb B^7})$. We define
				$$
					\tilde{\mathbf{L}}_0\mathbf u(\xi):=
						\begin{pmatrix}
						-\xi^j \partial_j-1&1
						\\
						\Delta &- \xi^j \partial_j-4
					\end{pmatrix}
					\begin{pmatrix}
						u_1(\xi)
						\\
						u_2(\xi)
					\end{pmatrix}			
				$$
			for $\xi\in\mathbb B^7$ where $\Delta=\partial^j\partial_j$ is the Laplacian on $\mathbb R^7$. With the domain $\mathcal D(\tilde{\mathbf L}_0):=C^\infty_\text{rad}(\overline{\mathbb B^7})\times C^\infty_\text{rad}(\overline{\mathbb B^7})$, the unbounded operator $\big(\tilde{\mathbf L}_0,\mathcal D(\tilde{\mathbf L}_0)\big)$ is densely-defined on $\mathcal H$. Writing $\tilde{\mathbf{L}}_0\mathbf u$ in terms of radial representatives gives exactly the first term on the right-hand side of Equation \eqref{first order eqn}. From \cite{CMS23}, we recall that $\big(\tilde{\mathbf L}_0,\mathcal D(\tilde{\mathbf L}_0)\big)$ is closable in $\mc H$ with closure denoted by $\big(\mathbf L_0,\mathcal D(\mathbf L_0)\big)$. Moreover, we define
				\begin{equation}
					\mathbf{L}'\mathbf{u}(\xi):=
						\begin{pmatrix}
						0
						\\
						V_1(|\xi|)u_1(\xi)+ V_2(|\xi|)\big(|\xi|^2u_2(\xi)-\xi^j\partial_j u_1(\xi)\big)
					\end{pmatrix}, \label{perturbation}
				\end{equation}
			with $V_1, V_2 \in C_e^\infty[0,1]$ defined in \eqref{Def_V1} and \eqref{V2_Def} respectively. Note that $\mathbf{L}'$ extends to a bounded operator on $\mathcal H$ which, by an abuse of notation, we again denote by $\mathbf L'$. From \cite{CMS23}, it follows that the operator $(\mathbf L, \mathcal D(\mathbf L))$, with $\mathbf L := \mathbf L_0+\mathbf L'$ and $\mathcal D(\mathbf L)=\mathcal D(\mathbf L_0) \subset \mc H$, is closed. Moreover, we have the following property regarding its spectrum
				\begin{proposition}[Section 3.3 of \cite{CMS23}] \label{spectrum of evolution}
					There is an $\omega_0>0$ such that 
						$$
							\sigma(\mb L)\subseteq\{\lambda\in\mathbb{C}:\Re\lambda \leq -\omega_0 \}\cup\{1\}.
						$$
					Furthermore, $1$ is an eigenvalue with a one-dimensional eigenspace, i.e., $\ker(1-\mb L)=\langle\mathbf g\rangle$ with $\mathbf g\in C^\infty_\text{rad}(\overline{\mathbb B^7}) \times C^\infty_\text{rad}(\overline{\mathbb B^7})\setminus\{\mathbf 0\}$. Furthermore, there exists a projection $\mathbf P$ which satisfies $\range\mathbf P=\langle\mathbf g\rangle$.
				\end{proposition}
			
			We remark that the existence of the eigenvalue $1$ and eigenfunction $\mathbf g$ is due to the time translation symmetry of the strong field equation \eqref{sf skyrme eqn}. That the algebraic multiplicity is one is standard and follows from ODE arguments. The projection $\mathbf P$ is simply the Riesz projection associated to the eigenvalue $1$. The precise form of $\mathbf g$ is explicit and can be found in \cite{CMS23}, however, it plays no important role in the current work and so we do not include it. On the other hand, the existence of $\omega_0>0$ such that the rest of the spectrum is contained to the left of $-\omega_0$ is nontrivial and depends crucially on the precise structure of $\mathbf L'$ as discussed in Section \ref{Outline of the proof}.
			
			The properties of $(\mathbf L, \mathcal D(\mathbf L))$ which allow us to formulate a weak version of Equations \eqref{first order eqn} and \eqref{Initial_Data_Trans} are the following.
				\begin{theorem}[Theorem 3.3 of \cite{CMS23}]\label{linear theory}
					The projection $\mathbf P$ commutes with $\big(\mathbf{S}(\tau)\big)_{\t\geq0}$. Furthermore, there exists $\omega\in(0,\frac{1}{2})$ and $C\geq1$ such that
						\begin{align*}
							\mathbf{S}(\tau)\mathbf{P}\mathbf{u}&=e^\tau\mathbf{P}\mathbf{u}\text{ and}
							\\
							\|\mathbf S(\t)(1-\mathbf P)\mathbf u\|_\mathcal H&\leq Ce^{-\omega\t}\|(1-\mathbf P)\mathbf u\|_\mathcal H
						\end{align*}
					for any $\mathbf u\in\mathcal H$ and all $\tau\geq0$.
				\end{theorem}
	
			For the nonlinear problem, we restrict our attention to the real-valued subspace of $\mc H^k$. Given $\d>0$ and $k\in\mathbb N$, we define
				$$
					\mathcal B_\delta^k:=\{\mathbf u\in\mathcal H^k:\|\mathbf u\|_{\mathcal H^k}\leq\delta\}.
				$$
			If $k=5$, then we will simply write $\mc B_\d:=\mc B_\d^5$. We also define the nonlinear expressions
				$$
					\mathbf{N}(\mathbf{u})(\xi):=
						\begin{pmatrix}
							0
							\\
							N\big(|\xi|u_1(\xi),\xi^j\partial_j u_1(\xi),|\xi|u_2(\xi),|\xi|\big)
						\end{pmatrix}
				$$
			and
				$$
					\mathbf{G}_{\lambda,T}(\tau,\mathbf{u})(\xi):=
						\begin{pmatrix}
							0
							\\
							\lambda^2T^2e^{-2\tau}\mathcal G_{\lambda,T}\big(|\xi|u_1(\xi),\xi^j\partial_j u_1(\xi),|\xi|u_2(\xi),\tau,|\xi|\big)
						\end{pmatrix}.
				$$
			In Section \ref{Nonlinear estimates}, we will show that $\mathbf G_{\lambda,T}$ defines a uniformly in time Lipschitz mapping on sufficiently small balls in $\mc H^k$. From \cite{CMS23}, we have the following local Lipschitz estimate for $\mathbf N$ on $\mc B_\d^k$.
				\begin{proposition}[Proposition 4.1 of \cite{CMS23}]\label{locally lipschitz nonlinearity estimate}
					Let $k\in\mathbb N$ with $k\geq5$. There exists $\d_0>0$ sufficiently small such that for any $\d\in(0,\d_0]$, the map $\mathbf N:\mathcal B_\d^k\to\mathcal H^k$ is defined and satisfies the local Lipschitz bound
						$$
							\|\mathbf N(\mathbf u)-\mathbf N(\mathbf v)\|_{\mathcal H^k}\lesssim_k\big(\|\mathbf u\|_{\mathcal H^k}+\|\mathbf v\|_{\mathcal H^k}\big)\|\mathbf u-\mathbf v\|_{\mathcal H^k}.
						$$
				\end{proposition}
	
			Furthermore, we denote the blowup profile by
				$$
					\mb U(\xi):=
					\begin{pmatrix}
						U_1(|\xi|)
						\\
						U_2(|\xi|)
					\end{pmatrix}.
				$$ 
			For functions $\mb v=(v_1,v_2)\in C^1_\text{rad}(\overline{\mathbb B^7_R})\times C_\text{rad}(\overline{\mathbb B^7_R})$, $R>0$, we define the rescaling operator
				$$
					\mathcal R(\mb v,T)(\xi):=\begin{pmatrix}
	   					T v_1(T\xi) 
						\\
	  					T^2 v_2(T\xi) 
					\end{pmatrix}
				$$
			for $\xi\in\overline{\mathbb B^7}$. For $T$ in an interval containing $1$, we define the initial data operator as
				$$
					\Phi_0(\mb v,T)(\xi)=\mathcal R(\mb v,T)(\xi)+\mathcal R(\mb U,T)(\xi)-\mathcal R(\mb U,1)(\xi).
				$$
			Now, consider the Hilbert space
				$$
					\mc Y^k:=H^{k+1}_\text{rad}(\mathbb B^7_2)\times H^{k}_\text{rad}(\mathbb B^7_2)
				$$ 
			with the standard inner product and denote by $\mc B_{\mc Y^k}$ the unit ball in $\mc Y^k$. For simplicity, we write $\mc Y:=\mc Y^5$. From \cite{CMS23}, we recall the following mapping properties of the initial data operator.
			
			\begin{lemma}[Lemma 4.9 of \cite{CMS23}]\label{id operator}
				Let $k\geq5$ and $\d\in(0,\d_0]$ for $\d_0>0$ sufficiently small. The initial data operator $\Phi_0:\mc B_{\mc Y^k}\times[1-\d,1+\d]\to\mathcal H^k$ is Lipschitz continuous, i.e., 
					$$
						\|\Phi_0(\mb v, T)-\Phi_0(\mb w, T')\|_{\mc H^k}\lesssim_k\|\mb v-\mb w\|_{\mc Y^k}+| T- T'|
					$$
				for all $\mb v,\mb w\in\mc B_{\mc Y^k}$ and $ T, T'\in[1-\d,1+\d]$. Moreover, the initial data operator can also be viewed as a map $\Phi_0:\mc B_{\mc Y^{k-1}}\times[1-\d,1+\d]\to\mathcal H^k$ and satisfies the bound 
					$$
						\|\Phi_0(\mb v, T)\|_{\mc H^k}\lesssim_k1+\|\mb v\|_{\mc Y^{k-1}}.
					$$
				Lastly, if $\|\mb v\|_{\mc Y^k}\leq\d$, then
					$$
						\|\Phi_0(\mb v,T)\|_{\mathcal H^k}\lesssim_k\d
					$$
				 for all $T\in[1-\d,1+\d]$.
			\end{lemma}
		
		\subsubsection{The Abstract Initial Value Problem}
			Our main object of study is the following abstract initial value problem
				\begin{equation}
				\begin{cases}
					\partial_\t\Phi(\t)=\mathbf L\Phi(\t)+\mathbf N\big(\Phi(\t)\big)+\mathbf{G}_{\lambda,T}\big(\tau,\Phi(\t)\big),&\t>0
					\\
					\Phi(0)=\Phi_0(\mathbf v,T)
				\end{cases} \label{abstract ivp}
				\end{equation}
			posed on the space $\mathcal H$. The main theorem follows from showing that for each sufficiently small $\lambda$ and $\mathbf v$, there exists a unique $T$ and a unique, global, classical solution $\Phi$ which decays exponentially as $\tau\to\infty$. To establish this, we first appeal to the Duhamel formula and seek the existence of a unique, global strong solution $\Phi\in C([0,\infty),\mathcal H)$ and a unique $T$ close to $1$ such that
				\begin{equation}
					\Phi(\tau)=\mathbf S(\tau)\Phi_0(\mathbf v,T)+\int_0^\tau\mathbf S(\tau-s)\Big(\mathbf N\big(\Phi(s)\big)+\mathbf{G}_{\lambda,T}\big(s,\Phi(s)\big)\Big)ds \label{duhamel form}
				\end{equation}
			for all $\t\in[0,\infty)$ with $\Phi(\tau)$ exhibiting the same decay as the semigroup $\big(\mathbf S(\tau)\big)_{\tau\geq0}$. 
		
\section{Existence of strong solutions}\label{Existence of strong solutions}
	This section is devoted to solving Equation \eqref{duhamel form}. We begin by proving that for all $\lambda,T>0$, $\mathbf G_{\lambda,T}$ is a uniformly in time Lipschitz mapping on small balls in $\mathcal H^k$.
	
	\subsection{Nonlinear estimates}\label{Nonlinear estimates}
		The aim of this section is to prove the following bound and uniform Lipschitz estimate on $\mathbf G_{\lambda,T}$.
			\begin{proposition}\label{lipschitz nonlinearity estimate}
				Let $k\in\mathbb N$ with $k\geq5$ and $\d\in(0,\d_0]$ for $\d_0>0$ sufficiently small. Fix $0<\lambda_\text{max},T_\text{max}<\infty$. For all $\lambda\in(0,\lambda_\text{max}]$ and $T\in(0,T_\text{max}]$, the map $\mathbf G_{\lambda,T}:[0,\infty)\times\mathcal B_\d^k\to\mathcal H^k$ is defined and satisfies the bound
					\begin{equation}
						\sup_{\t\geq0}e^{2\t}\|\mathbf G_{\lambda,T}(\tau,\mathbf u)\|_{\mathcal H^k}\lesssim_k\lambda^2T^2. \label{GlT smallness}
					\end{equation}
				Moreover, for all $\lambda,\lambda'\in(0,\lambda_\text{max}]$ and $T,T'\in(0,T_\text{max}]$ and $\mathbf u,\mathbf v\in\mathcal B_\delta^k$, we have the Lipschitz estimates
					\begin{equation}
						\sup_{\tau\geq0}e^{2\tau}\|\mathbf G_{\lambda,T}(\tau,\mathbf u)-\mathbf G_{\lambda,T}(\tau,\mathbf v)\|_{\mathcal H^k}\lesssim_k\lambda^2T^2\|\mathbf u-\mathbf v\|_\mc H \label{GlT uvlip}
					\end{equation}
				and
					\begin{equation}
						\sup_{\tau\geq0}e^{2\tau}\|\mathbf G_{\lambda, T}(\tau,\mathbf u)-\mathbf G_{\lambda', T'}(\tau,\mathbf u)\|_{\mathcal H^k}\lesssim_k(\lambda T+\lambda' T')|\lambda T-\lambda' T'|. \label{GlT LTlip}
					\end{equation}
			\end{proposition}
		Note that the upper bounds on $\lambda$ and $T$ are arbitrary. Eventually, we will take $T$ to be close to $1$ and $\lambda$ close to $0$.
		
		\subsubsection{Proof of Proposition \ref{lipschitz nonlinearity estimate}}
			We begin by bringing $\mathbf G_{\lambda,T}(\mathbf u)(\tau,\xi)$ into a form which is straightforward to estimate. In particular, note that the dependence on $\lambda$, $T$, and $\tau$ always shows up in the form of $\lambda^2T^2e^{-2\tau}$. With that in mind, we define the auxiliary quantities
				$$
					\mathfrak g(x,\sigma,|\xi|)=\bigg(\sigma^2+\frac{4\sin^2(x)}{|\xi|^2}\bigg)^{-1}
				$$
				and
				$$
					\mathcal G\big(\zeta_1,\zeta_2,\zeta_3,\sigma,|\xi|\big)=\mathfrak g\big(|\xi|U_1(|\xi|)+\zeta_1,\sigma,|\xi|\big)G\big(|\xi|U_1(|\xi|)+\zeta_1,\Lambda U_1(|\xi|)+\zeta_2,|\xi|U_2(|\xi|)+\zeta_3,|\xi|\big)
				$$
			so that we have the equality
				$$
					\mathcal G_{\lambda,T}\big(\zeta_1,\zeta_2,\zeta_3,\tau,|\xi|\big)=\mathcal G\big(\zeta_1,\zeta_2,\zeta_3,\lambda Te^{-\tau},|\xi|\big).
				$$
			for appropriate choices of $\zeta_1,\zeta_2,\zeta_3$. According to \eqref{def:fsf}, this expression is not defined whenever
				$$
					|\xi|U_1(|\xi|)+\zeta_1=k\pi,\quad k\in\mathbb Z,\;\xi\in\overline{\mathbb B^7_2}.
				$$
			Note that $|\xi|U_1(|\xi|)\geq0$ for all $\xi\in\overline{\mathbb B^7_2}$ with equality if and only if $\xi=0$. Thus, it suffices to ensure
				$$
					0<|\xi|U_1(|\xi|)+\zeta_1<\pi,\quad k\in\mathbb Z,\;\xi\in\overline{\mathbb B^7_2}\setminus\{0\}.
				$$
			A direct calculation shows that $\||\cdot|U_1(|\cdot|)\|_{L^{\infty}(\mathbb B^7_2)}<\pi$. Upon requiring 
				$$
					|\zeta_1|\leq\frac{1}{2}\big(\pi-\||\cdot|U_1(|\cdot|)\|_{L^{\infty}(\mathbb B^7_2)}\big)=:A,
				$$
			we ensure that the function 
				$$
					(\zeta_1,\zeta_2,\zeta_3,\sigma,\xi)\in[-A,A]\times\R\times\R\times[0,\infty)\times\overline{\mathbb B^7_R}\mapsto\mathcal G\big(\zeta_1,\zeta_2,\zeta_3,\sigma,|\xi|\big)
				$$ 
			is defined for $R\in(0,2]$ and is smooth. In particular, by the Sobolev embedding $H^k(\mathbb B_R^7)\hookrightarrow L^\infty(\mathbb B_R^7)$ for $k\geq4$, it is possible to pick $\delta_0>0$ small enough and depending on $R\in(0,2]$ such that for all $\mathbf u=(u_1,u_2)\in C^\infty_\text{rad}(\overline{\mathbb B_R^7})\times C^\infty_\text{rad}(\overline{\mathbb B_R^7})$ with $\|\mathbf u\|_{\mathcal H_R^k}\leq\delta_0$, we have
				$$
					(\sigma,\xi)\in[0,\infty)\times\overline{\mathbb B_R^7}\mapsto\mathcal G\big(|\xi|u_1(\xi),\xi^j\partial_ju_1(\xi),|\xi|u_2(\xi),\sigma,|\xi|\big)
				$$
			defines a smooth, radial function. 
			
			By Taylor's theorem, we can expand in the $\zeta_1,\zeta_2,\zeta_3$ variables as follows
				\begin{align}
				\begin{split}
					\mathcal G&\big(\zeta_1,\zeta_2,\zeta_3,\sigma,|\xi|\big)=\mathcal G_0(\sigma,|\xi|)+\mathcal G_1(\sigma,|\xi|)\zeta_1+\mathcal G_2(\sigma,|\xi|)\zeta_2+\mathcal G_3(\sigma,|\xi|)\zeta_3
					\\
					&+\zeta_1^2\int_0^1(1-s)\partial_1^2\mathcal G\big(s\zeta_1,s\zeta_2,s\zeta_3,\sigma,|\xi|\big)ds+\zeta_2^2\int_0^1(1-s)\partial_2^2\mathcal G\big(s\zeta_1,s\zeta_2,s\zeta_3,\sigma,|\xi|\big)ds
					\\
					&+\zeta_3^2\int_0^1(1-s)\partial_1^2\mathcal G\big(s\zeta_1,s\zeta_2,s\zeta_3,\sigma,|\xi|\big)ds+2\zeta_1\zeta_2\int_0^1(1-s)\partial_1\partial_2\mathcal G\big(s\zeta_1,s\zeta_2,s\zeta_3,\sigma,|\xi|\big)ds
					\\
					&+2\zeta_2\zeta_3\int_0^1(1-s)\partial_2\partial_3\mathcal G\big(s\zeta_1,s\zeta_2,s\zeta_3,\sigma,|\xi|\big)ds+2\zeta_3\zeta_1\int_0^1(1-s)\partial_3\partial_1\mathcal G\big(s\zeta_1,s\zeta_2,s\zeta_3,\sigma,|\xi|\big)ds \label{Taylor expansion of nonlinearity}
				\end{split}
				\end{align}
			where we have set for notational convenience
				\begin{align*}
					&\mathcal G_0(\sigma,|\xi|)=\mathcal G\big(0,0,0,\sigma,|\xi|\big),
					\\
					&\mathcal G_j(\sigma,|\xi|)=\partial_j\mathcal G\big(0,0,0,\sigma,|\xi|\big),\quad j=1,2,3.
				\end{align*}
			We prove the estimates \eqref{GlT smallness}-\eqref{GlT LTlip} one order (in $\zeta_1,\zeta_2,\zeta_3$) at a time. 
			
			The zeroth-order term is precisely $\mathcal G_0\big(\sigma,|\xi|\big)$. Explicitly, we have
				$$
					\mathcal G_0\big(\sigma,|\xi|\big)=\frac{20 \left(105-42 |\xi|^2+|\xi|^4\right)}{\left(5-|\xi|^2\right)^{3/2} \left(64\left(5-|\xi|^2\right)+\sigma^2\left(5+3|\xi|^2\right)^2\right)},  \label{zeroth order coefficient}
				$$
			A direct calculation verifies that
				\begin{equation}
					\sup_{\sigma\geq0}\|\mathcal G_0\big(\sigma,|\cdot|\big)\|_{H^k(\mathbb B^7)}<\infty \label{zeroth order estimate}
				\end{equation}
			for all $k\in\mathbb N$. Thus, \eqref{GlT smallness} holds at zeroth-order. In particular, observe that this term does not appear in \eqref{GlT uvlip}. To establish \eqref{GlT LTlip}, we write
				\begin{align}
				\begin{split}
					\lambda^2T^2\mathcal G_0&\big(\lambda Te^{-\tau},|\xi|\big)-\lambda'^2T'^2\mathcal G_0\big(\lambda'T'e^{-\tau},|\xi|\big)
					\\
					=&\big(\lambda^2T^2-\lambda'^2T'^2\big)\mathcal G_0\big(\lambda Te^{-\tau},|\xi|\big)+\lambda'^2T'^2\Big(\mathcal G_0\big(\lambda Te^{-\tau},|\xi|\big)-\mathcal G_0\big(\lambda'T'e^{-\tau},|\xi|\big)\Big). \label{add and subtract}
				\end{split}
				\end{align}
			The desired estimate follows for the first term via \eqref{zeroth order estimate} and the difference of squares formula. For the second term, we use the fundamental theorem of calculus to write
				$$
					\mathcal G_0\big(\lambda Te^{-\tau},|\xi|\big)-\mathcal G_0\big(\lambda'T'e^{-\tau},|\xi|\big)=\int_{\lambda'T'e^{-\tau}}^{\lambda Te^{-\tau}}\partial_1\mathcal G_0\big(\sigma,|\xi|\big)d\sigma.
				$$
			Explicitly, we have
				$$
					\partial_1\mathcal G_0\big(\sigma,|\xi|\big)=-\frac{40 \left(5+3|\xi|^2\right)^2 \left(105-42 |\xi|^2+|\xi|^4\right)\sigma}{\left(5-|\xi|^2\right)^{3/2} \left(64\left(5-|\xi|^2\right)+\sigma^2\left(5+3|\xi|^2\right)^2\right)^2}.
				$$
			A direct calculation verifies that
				\begin{equation}
					\sup_{\sigma\geq0}\|\partial_1\mathcal G_0\big(\sigma,|\cdot|\big)\|_{H^k(\mathbb B^7)}<\infty. \label{zeroth order sigma estimate}
				\end{equation}
			Thus,
				$$
					\Big\|\lambda'^2T'^2\Big(\mathcal G_0\big(\lambda Te^{-\tau},|\cdot|\big)-\mathcal G_0\big(\lambda'T'e^{-\tau},|\cdot|\big)\Big)\Big\|_{H^k(\mathbb B^7)}\lesssim_k(\lambda T+\lambda'T')|\lambda T-\lambda'T'|
				$$
			upon which \eqref{GlT LTlip} follows at zeroth-order.
			
			Plugging in $(\zeta_1,\zeta_2,\zeta_3)=\big(|\xi|u_1(\xi),\xi^j\partial_ju_1(\xi),|\xi|u_2(\xi)\big)$ into \eqref{Taylor expansion of nonlinearity}, we see that the linear terms are
				$$
					|\xi|\mathcal G_1\big(\sigma,|\xi|\big)u_1(\xi)+\mathcal G_2\big(\sigma,|\xi|\big)\xi^j\partial_ju_1(\xi)+|\xi|\mathcal G_3\big(\sigma,|\xi|\big)u_2(\xi).
				$$
			Explicitly, we have
				$$
					|\xi|\mathcal G_1\big(\sigma,|\xi|\big)=-\frac{p_1(|\xi|)+\sigma^2p_2(|\xi|)}{\left(5-|\xi|^2\right)^2\left(64\left(5-|\xi|^2\right)+\sigma ^2\left(5+3|\xi|^2\right)^2 \right)^2},
				$$
				$$
					\mathcal G_2\big(\sigma,|\xi|\big)=\frac{2 \left(35-3|\xi|^2\right) \left(5+3|\xi|^2\right)}{\left(5-|\xi|^2\right) \left(64\left(5-|\xi|^2\right)+\sigma^2\left(5+3|\xi|^2\right)^2 \right)},
				$$
			and
				$$
					|\xi|\mathcal G_3\big(\sigma,|\xi|\big)=\frac{50\left(1-|\xi|^2\right)\left(5+3|\xi|^2\right) }{\left(5-|\xi|^2\right)\left(64\left(5-|\xi|^2\right)+\sigma^2\left(5+3|\xi|^2\right)^2 \right)}
				$$
			where
				$$
					p_1(\rho)=64 \left(27 \rho^8-2240 \rho^6+15550 \rho^4-25000 \rho^2-625\right)
				$$
			and
				$$
					p_2(\rho)=\left(5+3 \rho^2\right)^2 \left(23 \rho^6-45 \rho^4+2325 \rho^2-5375\right).
				$$
			Direct calculations verify that
				$$
					\sup_{\sigma\geq0}\Big(\||\cdot|\mathcal G_1\big(\sigma,|\cdot|\big)\|_{H^k(\mathbb B^7)}+\|\mathcal G_2\big(\sigma,|\cdot|\big)\|_{H^k(\mathbb B^7)}+\||\cdot|\mathcal G_3\big(\sigma,|\cdot|\big)\|_{H^k(\mathbb B^7)}\Big)<\infty
				$$
			for all $k\in\mathbb N$. In particular for $k\geq5$, the estimates \eqref{GlT smallness} and \eqref{GlT uvlip} follow at linear-order from the Banach algebra property of $H^{k-1}(\mathbb B^7)$. The estimate \eqref{GlT LTlip} follows by splitting the corresponding difference as in \eqref{add and subtract}, bounding the first term via the previous bound, and applying the same fundamental theorem of calculus argument applied at zeroth order for the second term combined with the fact that $\mathbf u\in\mathcal B_\delta^k$ and the Banach algebra property of $H^{k-1}(\mathbb B^7)$, $k\geq5$.
			
			For the remaining terms, observe that plugging in $(\zeta_1,\zeta_2,\zeta_3)=(0,0,0)$ yields $0$ (in particular, this is also true for the linear terms but provides no significant simplification). Thus, it suffices to prove \eqref{GlT uvlip}  and \eqref{GlT LTlip} at quadratic (and higher)-order to finish the proof of Proposition \ref{Nonlinear estimates}. 
			
			For notational convenience, we define
				$$
					\mathfrak g_1(\zeta_1,\zeta_2,\zeta_3,\sigma,\xi)=\zeta_1^2\int_0^1(1-s)\partial_1^2\mathcal G\big(s\zeta_1,s\zeta_2,s\zeta_3,\sigma,|\xi|\big)ds.
				$$
			By the fundamental theorem of calculus, we can write
				\begin{align*}
					\mathfrak g_1&\big(|\xi|u_1(\xi),\xi^j\partial_ju_1(\xi),|\xi|u_2(\xi),\sigma,\xi\big)-\mathfrak g_1\big(|\xi|v_1(\xi),\xi^j\partial_jv_1(\xi),|\xi|v_2(\xi),\sigma,\xi\big)
					\\
					&=\int_0^1\nabla\mathfrak g_1\big(\mathbf c(t),\sigma,\xi\big)\cdot\mathbf c'(t)dt
				\end{align*}
			where
				$$
					\mathbf c(t)=\begin{pmatrix}|\xi|u_1(\xi)+t(|\xi|v_1(\xi)-|\xi|u_1(\xi))\\\xi^j\partial_ju_1(\xi)+t(\xi^j\partial_jv_1(\xi)-\xi^j\partial_ju_1(\xi))\\|\xi|u_2(\xi)+t(|\xi|v_2(\xi)-|\xi|u_2(\xi))\end{pmatrix},\quad t\in[0,1]
				$$
			and $\nabla$ denotes the gradient with respect to the first three variables. To prove \eqref{GlT uvlip}, it suffices to prove 
				\begin{equation}
					\sup_{\sigma\geq0}\|\nabla\mathfrak g_1\big(\mathbf c(t),\sigma,\cdot\big)\|_{H^{k-1}(\mathbb B^7)}<\infty \label{goal inequality}
				\end{equation}
			for $t\in[0,1]$.
			
			To prove \eqref{goal inequality}, we first extend the domain of $\nabla\mathfrak g_1$ as follows. Fix three smooth cutoff functions $\chi_1:\R\to[0,1]$, $\chi_2:\mathbb R^7\to[0,1]$, and $\chi_3:\R\to[0,1]$ with the properties that 
				\begin{enumerate}
					\item $\chi_1(\zeta_1)=1$ for $|\zeta_1|\leq\frac{A}{2}$, $\chi_1(\zeta_1)=0$ for $|\zeta_1|\geq\frac{2A}{3}$, $\chi_1$ decreases smoothly in the transition region,
					\item $\chi_2(\xi)=1$ for $|\xi|\leq\frac{3}{2}$, $\chi_2(\xi)=0$ for $|\xi|\geq\frac{5}{3}$, $\chi_2$ decreases smoothly and radially in the transition region, and
					\item $\chi_3(\sigma)=1$ for $-\frac{1}{2}\leq s<\infty$, $\chi_3(\sigma)=0$ for $s\leq-1$, and $\chi_3$ decreases smoothly in the transition region.
				\end{enumerate}
			Now, consider the auxiliary function $\tilde{\mathfrak g_1}:\mathbb R\times\mathbb R\times\mathbb R\times\mathbb R\times\mathbb R^7\to\mathbb R$ by
				$$
					\tilde{\mathfrak g_1}(\zeta_1,\zeta_2,\zeta_3,\sigma,\xi)=\chi_1(\zeta_1)\chi_2(\xi)\chi_3(\sigma)\nabla\mathfrak g_1(\zeta_1,\zeta_2,\zeta_3,\sigma,\xi)
				$$
			for $(\zeta_1,\zeta_2,\zeta_3,\sigma,\xi)\in[-A,A]\times\R\times\R\times[0,\infty)\times\overline{\mathbb B^7_R}$ and $\tilde{\mathfrak g_1}(\zeta_1,\zeta_2,\zeta_3,\sigma,\xi)=0$ otherwise. A direct calculation verifies that $\tilde{\mathfrak g}_1\in C^\infty(\mathbb R\times\mathbb R\times\R\times\R\times\mathbb R^7)$, $\tilde{\mathfrak g}_1(0,0,0,\sigma,\xi)=0$ for all $\sigma\in\mathbb R$ and $\xi\in\mathbb R^7$, and $\partial^\alpha\tilde{\mathfrak g}_1\in L^\infty(\mathbb R\times\mathbb R\times\R\times\R\times\mathbb R^7)$ for all $\alpha\in\mathbb N^{11}_0$ with $|\alpha|\leq k$. By Moser's inequality (see, e.g., Theorem 6.4.1 of \cite{R12}), there exists a continuous function $\gamma:[0,\infty)\to[0,\infty)$ such that
				$$
					\sup_{\sigma\in\mathbb R}\|\tilde{\mathfrak g}_1\big(f_1,\xi^j\partial_jf_1,f_2,\sigma,\cdot\big)\|_{H^{k-1}(\mathbb R^7)}\leq\gamma\big(\|\mathbf f\|_{H^{k}(\mathbb R^7)\times H^{k-1}(\mathbb R^7)}\big)\|\mathbf f\|_{H^{k}(\mathbb R^7)\times H^{k-1}(\mathbb R^7)}
				$$
			for all $\mathbf f=(f_1,f_2)\in H^{k}(\mathbb R^7)\times H^{k-1}(\mathbb R^7)$. Denoting by $\mathcal E$ a Sobolev extension as in Lemma 2.4 of \cite{BDS19}, we then obtain
				$$
					\sup_{\sigma\geq0}\|\nabla\mathfrak g_1\big(\mathbf c(t),\sigma,\cdot\big)\|_{H^{k-1}(\mathbb B^7)}\leq\sup_{\sigma\in\mathbb R}\|\tilde{\mathfrak g}_1\big(\mathcal E\mathbf c(t),\sigma,\cdot\big)\|_{H^{k-1}(\mathbb R^7)}<\infty
				$$
			for all $t\in[0,1]$ and $\mathbf u,\mathbf v\in\mathcal B_\delta^k$. This concludes the proof of \eqref{GlT smallness} and \eqref{GlT uvlip} for this first quadratic term. The same argument with insignificant aesthetic changes yields \eqref{GlT smallness} and \eqref{GlT uvlip} for the remaining quadratic terms.
			
			To obtain \eqref{GlT LTlip} for the first quadratic term, we split the difference
				$$
					\mathfrak g_1\big(|\xi|u_1(\xi),\xi^j\partial_ju_1(\xi),|\xi|u_2(\xi),\lambda Te^{-\tau},\xi\big)-\mathfrak g_1\big(|\xi|u_1(\xi),\xi^j\partial_ju_1(\xi),|\xi|u_2(\xi),\lambda'T'e^{-\tau},\xi\big)
				$$ 
			as in \eqref{add and subtract}, apply the inequality just obtained in the special case $\mathbf v=\mathbf 0$ to the first term in this splitting, and repeat the previous Moser inequality argument for the quantity
				$$
					\partial_\sigma\mathfrak g_1\big(|\xi|u_1(\xi),\xi^j\partial_ju_1(\xi),|\xi|u_2(\xi),\sigma,\xi\big).
				$$
			Again, obtaining \eqref{GlT LTlip} for the remaining quadratic terms follows mutatis mutandis.

		\subsection{Existence of modified strong solutions} \label{Existence of modified strong solutions}
			We turn our attention to solving Equation \eqref{duhamel form}. At first pass, we solve
				\begin{equation}
					\Phi(\tau)=\mathbf S(\tau)\mb u+\int_0^\tau\mathbf S(\tau-s)\Big(\mathbf N\big(\Phi(s)\big)+\mathbf{G}_{\lambda,T}\big(s,\Phi(s)\big)\Big)ds \label{duhamel form general data}
				\end{equation}
			for $\lambda,T>0$ and bounded from above and $\mb u\in\mathcal B_\delta$ for any $\delta\leq\delta_0$ with $\d_0$ small enough so that both Propositions \ref{locally lipschitz nonlinearity estimate} and \ref{lipschitz nonlinearity estimate} hold. This can be achieved by reframing Equation \eqref{duhamel form general data} as a fixed-point problem on the Banach space
				$$
					\mathcal X:=\{\Phi\in C([0,\infty),\mathcal H):\|\Phi\|_\mathcal X:=\sup_{\t>0}e^{\omega\t}\|\Phi(\t)\|_\mathcal H<\infty\}
				$$
			for $\omega>0$ as in Theorem \ref{linear theory}. However, due to $1\in\sigma_p(\mathbf L)$, it is not possible to prove the existence of a fixed point in the space $\mathcal X$ for arbitrary small data $\mb u$. To remedy this, we first consider a modified problem motivated by the Lyapunov-Perron method from dynamical systems theory. Given $\Phi\in\mc X$ and $\mathbf u\in\mc B_\d$, we introduce a correction term
				$$
					\mathbf C_{\lambda,T}(\Phi,\mathbf u):=\mathbf P\bigg(\mathbf u+\int_0^\infty e^{-s}\Big(\mathbf N\big(\Phi(s)\big)+\mathbf{G}_{\lambda,T}\big(s,\Phi(s)\big)\Big)ds\bigg)
				$$
			and consider the modified equation
				\begin{equation}
					\Phi(\t)=\mathbf S(\tau)\big(\mathbf u-\mathbf C_{\lambda,T}(\Phi,\mathbf u)\big)+\int_0^\tau\mathbf S(\tau-s)\Big(\mathbf N\big(\Phi(s)\big)+\mathbf{G}_{\lambda,T}\big(s,\Phi(s)\big)\Big)ds. \label{modified Duhamel}
				\end{equation}
			We will first show that for any $\lambda>0$ sufficiently small and $T>0$ bounded, there exists a unique solution of Equation \eqref{modified Duhamel} within the space $\mathcal X$.
		
			\begin{proposition} \label{mod wp}
				Let $T_\text{max}>0$. For all sufficiently large $c>0$, $\delta\in(0,\delta_0]$ for sufficiently small $\delta_0>0$, the following holds. For any $\mathbf u\in\mathcal H$ satisfying $\|\mathbf u\|_{\mathcal H}\leq\frac{\delta}{c}$, $|\lambda|\leq2\d$, and $T\in(0,T_\text{max})$ there exists a unique solution $\Phi_{\mathbf u,\lambda,T}\in C([0,\infty),\mathcal H)$ of Equation \eqref{modified Duhamel} that satisfies $\|\Phi_{\mathbf u,\lambda,T}(\t)\|_{\mathcal H}\leq\delta e^{-\omega\t}$ for all $\t\geq0$. Furthermore, the map $(\mathbf u,\lambda,T)\mapsto\Phi_{\mathbf u,\lambda,T}$ is Lipschitz from $\mc B_{\delta/c}\times(0,2\delta]\times(0,T_\text{max})$ to $\mathcal X$. More precisely,
					$$
						\|\Phi_{\mathbf u,\lambda,T}-\Phi_{\mathbf v,\lambda',T'}\|_{\mathcal X}\lesssim\|\mathbf u-\mathbf v\|_{\mathcal H}+|\lambda-\lambda'|+\delta|T-T'|
					$$
				for all $\mathbf u,\mathbf v\in\mathcal B_{\delta/c}$, $\lambda,\lambda'\in(0,2\delta]$ and $T,T'\in(0,T_\text{max})$.
			\end{proposition}
			\begin{proof}
				Consider the closed ball
					$$
						\mathcal X_\delta:=\{\Phi\in C([0,\infty),\mathcal H):\|\Phi\|_\mathcal X\leq\delta\}
					$$
				and the map
					$$
						\mathbf K_{\mathbf u,\lambda,T}(\Phi)(\t):=\mathbf S(\tau)\big(\mathbf u-\mathbf C_{\lambda,T}(\Phi,\mathbf u)\big)+\int_0^\tau\mathbf S(\tau-s)\Big(\mathbf N\big(\Phi(s)\big)+\mathbf{G}_{\lambda,T}\big(s,\Phi(s)\big)\Big)ds.
					$$
				We show that for each $(\mathbf u,\lambda,T)\in\mathcal B_{\delta/c}\times(0,2\delta]\times(0,T_\text{max})$, the map $\mathbf K_{\mathbf u,\lambda,T}:\mathcal X_\delta\to\mathcal X_\delta$ is a well-defined contraction. 
				
				First, observe that by Theorem \ref{linear theory}, we have
					$$
						\mathbf P\mathbf K_{\mathbf u,\lambda,T}(\Phi)(\t)=-\int_\t^\infty e^{\t-s}\mathbf P\Big(\mathbf N\big(\Phi(s)\big)+\mathbf{G}_{\lambda,T}\big(s,\Phi(s)\big)\Big)ds.
					$$
				From Propositions \ref{locally lipschitz nonlinearity estimate} and \ref{lipschitz nonlinearity estimate} and the fact that $\mathbf N(\mathbf0)=\mathbf0$, we have the estimate
					\begin{align*}
						\|\mathbf P\mathbf K_{\mathbf u,\lambda,T}(\Phi)(\t)\|_{\mathcal H}&\lesssim e^{\t}\int_\t^\infty e^{-s}\Big(\|\Phi(s)\|_{\mathcal H}^2+\lambda^2T^2e^{-2s}\Big)ds
						\\
						&\lesssim e^{\t}\int_\t^\infty e^{-s}\Big(\|\Phi\|_{\mathcal X}^2e^{-2\omega s}+\lambda^2T^2e^{-2s}\Big)ds\lesssim \delta^2e^{-2\omega\t}.
					\end{align*}
				From the fact that $\mathbf P$ is a projection, we have $(1-\mathbf P)\mathbf C_{\lambda,T}(\Phi,\mathbf u)=\mathbf 0$. This implies
					$$
						(1-\mathbf P)\mathbf K_{\mathbf u,\lambda,T}(\Phi)(\t)=\mathbf S(\t)(1-\mathbf P)\mathbf u+\int_{0}^\t\mathbf S(\t-s)(1-\mathbf P)\Big(\mathbf N\big(\Phi(s)\big)+\mathbf{G}_{\lambda,T}\big(s,\Phi(s)\big)\Big)ds.
					$$
				By Theorem \ref{linear theory}, we obtain
					\begin{align*}
						\|(1-\mathbf P)&\mathbf K_{\mathbf u,\lambda,T}(\Phi)(\t)\|_{\mathcal H}
						\\
						&\lesssim e^{-\omega \t}\|(1-\mathbf P)\mathbf u\|_{\mathcal H}+\int_{0}^\t e^{-\omega (\t-s)}\Big(\|\mathbf N\big(\Phi(s)\big)\|_{\mathcal H}+\|\mathbf{G}_{\lambda,T}\big(s,\Phi(s)\big)\|_{\mathcal H}\Big)ds
						\\
						&\lesssim\frac{\delta}{c}e^{-\omega\t}+e^{-\omega\t}\int_{0}^\t e^{\omega s}\Big(\|\Phi(s)\|_{\mathcal H}^2+\lambda^2T^2e^{-2s}\Big)ds
						\\
						&\lesssim\frac{\delta}{c}e^{-\omega\t}+e^{-\omega\t}\int_{0}^\t e^{\omega s}\Big(\|\Phi\|_{\mathcal X}^2e^{-2\omega s}+\lambda^2T^2e^{-2s}\Big)ds
						\\
						&\lesssim\frac{\delta}{c}e^{-\omega\t}+\delta^2e^{-\omega\t}
					\end{align*}
				for all $\t\geq0$. Thus, by taking $\delta_0$ smaller if necessary and $c$ sufficiently large, we can ensure
					$$
						\|\mathbf K_{\mathbf u,\lambda,T}(\Phi)(\t)\|_{\mathcal H}\leq\delta e^{-\omega \t}.
					$$
				Consequently, we see that $\mathbf K_{\mathbf u,\lambda,T}:\mathcal X_\delta\to\mathcal X_\delta$ for all such $\lambda$ and $T$. 
				
				We claim that $\mathbf K_{\mathbf u,\lambda,T}$ is a contraction map. Given $\Phi,\Psi\in\mathcal X_\delta$, observe that
					\begin{align*}
						\mathbf P\mathbf K_{\mathbf u,\lambda,T}(\Phi)(\t)-\mathbf P\mathbf K_{\mathbf u,\lambda,T}(\Psi)(\t)=&-\int_\t^\infty e^{\t-s}\mathbf P\Big(\mathbf N\big(\Phi(s)\big)-\mathbf N\big(\Psi(s)\big)\Big)ds
						\\
						&-\int_\t^\infty e^{\t-s}\mathbf P\Big(\mathbf G_{\lambda,T}\big(s,\Phi(s)\big)-\mathbf G_{\lambda,T}\big(s,\Psi(s)\big)\Big)ds.
					\end{align*}
				By Propositions \ref{locally lipschitz nonlinearity estimate} and \ref{lipschitz nonlinearity estimate}, we have that 
					\begin{align*}
						\|\mathbf P\mathbf K_{\mathbf u,\lambda,T}(\Phi)(\t)&-\mathbf P\mathbf K_{\mathbf u,\lambda,T}(\Psi)(\t)\|_{\mathcal H}	
						\\
						\lesssim&e^{\t}\int_\t^\infty e^{-s}\big(\|\Phi(s)\|_{\mathcal H}+\|\Psi(s)\|_{\mathcal H}\big)\|\Phi(s)-\Psi(s)\|_{\mathcal H}ds
						\\
						&+e^\tau\int_\tau^\infty e^{-s}\lambda^2T^2e^{-2s}\|\Phi(s)-\Psi(s)\|_{\mathcal H}ds
						\\
						\lesssim&\delta\|\Phi-\Psi\|_{\mathcal X}e^{\t}\int_\t^\infty\big(e^{-s-2\omega s}+e^{-3s-\omega s}\big)ds
						\\
						\lesssim&\delta e^{-2\omega\t}\|\Phi-\Psi\|_{\mathcal X}.
					\end{align*}
				Furthermore,
					\begin{align*}
						(1-\mathbf P)\mathbf K_{\mathbf u,\lambda,T}&(\Phi)(\t)-(1-\mathbf P)\mathbf K_{\mathbf u,\lambda,T}(\Psi)(\t)
						\\
						=&\int_0^\t\mathbf S(\t-s)(1-\mathbf P)\Big(\mathbf N\big(\Phi(s)\big)-\mathbf N\big(\Psi(s)\big)\Big)ds
						\\
						&+\int_0^\t\mathbf S(\t-s)(1-\mathbf P)\Big(\mathbf G_{\lambda,T}\big(s,\Phi(s)\big)-\mathbf G_{\lambda,T}\big(s,\Psi(s)\big)\Big)ds.
					\end{align*}
				By Theorem \ref{linear theory} and Propositions \ref{locally lipschitz nonlinearity estimate} and \ref{lipschitz nonlinearity estimate}, we obtain
					\begin{align*}
						\|(1-\mathbf P)\mathbf K_{\mathbf u,\lambda,T}(\Phi)(\t)&-(1-\mathbf P)\mathbf K_{\mathbf u,\lambda,T}(\Psi)(\t)\|_{\mathcal H}
						\\
						\lesssim&\int_0^\t e^{-\omega (\t-s)}\big(\|\Phi(s)\|_{\mathcal H}+\|\Psi(s)\|_{\mathcal H}\big)\|\Phi(s)-\Psi(s)\|_{\mathcal H}ds
						\\
						&+\int_0^\t e^{-\omega (\t-s)}\lambda^2T^2e^{-2s}\|\Phi(s)-\Psi(s)\|_{\mathcal H}ds
						\\
						\lesssim&\delta\|\Phi-\Psi\|_{\mathcal X}e^{-\omega \t}\int_0^\t\big(e^{-\omega s}+e^{-2s}\big)ds
						\lesssim\delta e^{-\omega \t}\|\Phi-\Psi\|_{\mathcal X}.
					\end{align*}
				Thus,
					$$
						\|\mathbf K_{\mathbf u,\lambda,T}(\Phi)-\mathbf K_{\mathbf u,\lambda,T}(\Psi)\|_{\mathcal X}\lesssim\delta\|\Phi-\Psi\|_{\mathcal X}
					$$
				and by considering smaller $\delta_0$ if necessary, we see that $\mathbf K_{\mathbf u,\lambda,T}$ is a contraction on $\mathcal X_\delta$. The Banach fixed point theorem implies the existence of a unique fixed point $\Phi_{\mathbf u,\lambda,T}\in\mathcal X_\delta$ of $\mathbf K_{\mathbf u,\lambda,T}$. 
			
				We now show that the solution map $(\mathbf u,\lambda,T)\mapsto\Phi_{\mathbf u,\lambda,T}$ is Lipschitz. Observe that
					\begin{align*}
						\|\Phi_{\mathbf u,\lambda,T}-&\Phi_{\mathbf v,\lambda',T'}\|_{\mathcal X}=\|\mathbf K_{\mathbf u,\lambda,T}(\Phi_{\mathbf u,\lambda,T})-\mathbf K_{\mathbf v,\lambda',T'}(\Phi_{\mathbf v,\lambda',T'})\|_{\mathcal X}
						\\
						\leq&\|\mathbf K_{\mathbf u,\lambda,T}(\Phi_{\mathbf u,\lambda,T})-\mathbf K_{\mathbf u,\lambda,T}(\Phi_{\mathbf v,\lambda',T'})\|_{\mathcal X}+\|\mathbf K_{\mathbf u,\lambda,T}(\Phi_{\mathbf v,\lambda',T'})-\mathbf K_{\mathbf v,\lambda',T'}(\Phi_{\mathbf v,\lambda',T'})\|_{\mathcal X}
						\\
						\leq&\delta\|\Phi_{\mathbf u,\lambda,T}-\Phi_{\mathbf v,\lambda',T'}\|_\mc X+\|\mathbf K_{\mathbf u,\lambda,T}(\Phi_{\mathbf v,\lambda',T'})-\mathbf K_{\mathbf v,\lambda',T'}(\Phi_{\mathbf v,\lambda',T'})\|_{\mathcal X}.
					\end{align*}
				A direct calculation shows
					\begin{align*}
						\mathbf K_{\mathbf u,\lambda,T}&(\Phi_{\mathbf v,\lambda',T'})(\tau)-\mathbf K_{\mathbf v,\lambda',T'}(\Phi_{\mathbf v,\lambda',T'})(\tau)
						\\
						=&\mathbf S(\t)(1-\mathbf P)(\mathbf u-\mathbf v)-\mathbf Pe^\tau\int_\tau^\infty e^{-s}\Big(\mathbf G_{\lambda,T}\big(s,\Phi_{\mathbf v,\lambda',T'}(s)\big)-\mathbf G_{\lambda',T'}\big(s,\Phi_{\mathbf v,\lambda',T'}(s)\big)\Big)ds
						\\
						&+\int_0^\tau\mathbf S(\tau-s)(1-\mathbf P)\Big(\mathbf G_{\lambda,T}\big(s,\Phi_{\mathbf v,\lambda',T'}(s)\big)-\mathbf G_{\lambda',T'}\big(s,\Phi_{\mathbf v,\lambda',T'}(s)\big)\Big)ds.
					\end{align*}
				Theorem \ref{linear theory} and Proposition \ref{lipschitz nonlinearity estimate} yield
					$$
						\|\mathbf K_{\mathbf u,\lambda,T}(\Phi_{\mathbf v,\lambda',T'})(\tau)-\mathbf K_{\mathbf v,\lambda',T'}(\Phi_{\mathbf v,\lambda',T'})(\tau)\|_{\mathcal H}\lesssim e^{-\omega \t}\big(\|\mathbf u-\mathbf v\|_{\mathcal H}+|\lambda T-\lambda'T'|\big).
					$$
				Thus, we have
					$$
						\|\Phi_{\mathbf u,T}-\Phi_{\mathbf v,T'}\|_{\mathcal X}\lesssim\delta\|\Phi_{\mathbf u,T}-\Phi_{\mathbf v,T'}\|_{\mathcal X}+\|\mathbf u-\mathbf v\|_{\mathcal H}+|\lambda-\lambda'|+\delta|T-T'|.
					$$
				Again, considering smaller $\delta_0$ if necessary yields the result.
			\end{proof}
			Now, instead of evolving arbitrary small initial data, we evolve the initial data $\Phi_0(\mathbf v,T)$ for $\mathbf v\in\mathcal Y$. By taking $T$ sufficiently close to $1$ and $\mathbf v$ sufficiently small in $\mc Y$, we are able to guarantee the necessary smallness of $\Phi_0(\mathbf v,T)$ as follows. 
			\begin{corollary}\label{id operator mod solution}
				Let $M>0$ be sufficiently large, $\d\in(0,\d_0]$ for sufficiently small $\delta_0>0$, and $T_\text{max}>0$. For all $\lambda\in(0,\d]$, $ T\in[1-\frac{\d}{M},1+\frac{\d}{M}]$, $ T'\in(0,T_\text{max})$, and $\mathbf v\in\mathcal Y$ with $\|\mathbf v\|_\mathcal Y\leq\frac{\d}{M}$ there exists a unique solution $\Phi_{\Phi_0(\mathbf v, T),\lambda, T'}\in C([0,\infty),\mathcal H)$ of
					\begin{align}
					\begin{split}
						\Phi(\t)=\mathbf S(\tau)&\Big[\Phi_0(\mathbf v, T)-\mathbf C_{\lambda, T'}\big(\Phi,\Phi_0(\mathbf v, T)\big)\Big]
						\\
						&+\int_0^\tau\mathbf S(\tau-s)\Big(\mathbf N\big(\Phi(s)\big)+\mathbf{G}_{\lambda, T'}\big(s,\Phi(s)\big)\Big)ds \label{T1T2 Duhamel}
					\end{split}
					\end{align}
				that satisfies $\|\Phi_{\Phi_0(\mathbf v, T),\lambda, T'}(\t)\|_{\mathcal H}\leq\delta e^{-\omega\t}$ for all $\t\geq0$. Moreover, 
					\begin{equation}
						\|\Phi_{\Phi_0(\mathbf v_1, T_1),\lambda_1, T_1'}-\Phi_{\Phi_0(\mathbf v_2, T_2),\lambda_2, T_2'}\|_\mathcal X\lesssim\|\mathbf v_1-\mathbf v_2\|_\mathcal Y+|T_1-T_2|+|\lambda_1-\lambda_2|+\delta|T_1'- T_2'| \label{Lipschitz estimate}
					\end{equation}
				for all $\mathbf v_1,\mathbf v_2\in\mathcal Y$ satisfying the smallness condition, $ T_1,T_2\in[1-\frac{\d}{M},1+\frac{\d}{M}]$, $\lambda_1,\lambda_2\in(0,\delta]$, and $T_1',T_2'\in(0,T_\text{max})$.
			\end{corollary}
			\begin{proof}
				From Lemma \ref{id operator}, we have that $\|\Phi_0(\mathbf v,T)\|_\mc H\lesssim\frac{\delta}{M}$. For $M$ sufficiently large, we can ensure $\|\Phi_0(\mathbf v,T)\|_\mc H\leq\frac{\delta}{c}$. Proposition \ref{mod wp} then yields the first claim. The second claim follows from Lipschitz continuity of the data to solution map from Proposition \ref{mod wp} and Lipschitz continuity of the initial data operator from Lemma \ref{id operator}.
			\end{proof}
			
		\subsection{Variation of the blowup time}\label{Reconnecting to the Physical Problem}
			By taking $T=T'\in[1-\frac{\d}{C},1+\frac{\d}{C}]$ in Equation \eqref{T1T2 Duhamel} and $C>0$ sufficiently large, we are able to show that there exists a unique $T\in[1-\frac{\d}{C},1+\frac{\d}{C}]$ such that the corresponding solution $\Phi_{\Phi_0(\mathbf v,T),\lambda,T}$ of Equation \eqref{T1T2 Duhamel} is in fact the solution of Equation \eqref{duhamel form}.
			
			\begin{lemma} \label{correction is zero}
				Let $C>0$ be sufficiently large and $\d\in(0,\d_0]$ for sufficiently small $\delta_0>0$. For all $\lambda\in(0,\d]$ and $\mathbf v\in\mathcal Y$ with $\|\mb v\|_\mc Y\leq\frac{\d}{C^2}$, there exists a unique $T\in[1-\frac{\d}{C},1+\frac{\d}{C}]$ and a unique $\Phi\in\mathcal X_\d$ which satisfies
					\begin{equation}
						\Phi(\t)=\mathbf S(\t)\Phi_0(\mathbf v,T)+\int_0^\t\mathbf S(\t-s)\Big(\mathbf N\big(\Phi(s)\big)+\mathbf{G}_{\lambda,T}\big(s,\Phi(s)\big)\Big)ds \label{duhamel with id operator}
					\end{equation}
				for all $\t>0$. Moreover, $T$ depends Lipschitz continuously on $\lambda$ and $\mathbf v$, i.e.,
					$$
						|T(\lambda,\mb v)-T(\lambda',\mb w)|\lesssim\|\mb v-\mb w\|_\mc Y+|\lambda-\lambda'|
					$$
				for all $\mb v,\mb w\in\mc Y$ and $\lambda,\lambda'\in(0,\d]$ as above.
			\end{lemma}
			\begin{proof}
				Take $C>M$ sufficiently large. For convenience of notation, we write
					$$
						\Phi_{\mathbf v,\lambda,T}=\Phi_{\Phi_0(\mathbf v,T),\lambda,T}
					$$
				to denote the solution obtained as in Corollary \ref{id operator mod solution}. The claim follows from the existence of a unique $T\in[1-\frac{\d}{C},1+\frac{\d}{C}]$ such that 
					$$
						\mathbf C_{\mathbf v,\lambda,T}:=\mathbf C_{\lambda,T}\big(\Phi_{\mathbf v,\lambda,T},\Phi_0(\mathbf v,T)\big)=\mathbf 0.
					$$ 
				Since $\rg\mathbf P=\langle\mathbf g\rangle$, this is equivalent to
					\begin{equation}
						\big(\mathbf C_{\mathbf v,\lambda,T}|\mathbf g\big)_\mathcal H=0. \label{inner prod is zero}
					\end{equation}
				We proceed by separating the linear and higher-order terms (in $T$) in Equation \eqref{inner prod is zero} and solving this equation using the Banach fixed point theorem. 
				
				By Taylor expansion, we have
					$$
						\mc R(\mb U,T)-\mc R(\mb U,1)=\kappa(T-1)\mathbf g+\mb R(T)
					$$
				for some $\kappa\in\mathbb R\setminus\{0\}$ and $\mb R(T)$ denoting the second-order remainder term and satisfying
					$$
						\|\mb R( T)-\mb R( T')\|_\mc H\lesssim\frac{\delta}{C}| T- T'|.
					$$
				With this, we write the initial data operator as
					$$
						\Phi_0(\mb v,T)=\mc R(\mb v,T)+\kappa(T-1)\mathbf g+\mb R(T).
					$$
				We write $T=1+\beta$ for $\beta\in[-\frac{\d}{C},\frac{\d}{C}]$ and define the following quantity
					$$
						\mb\Sigma_{\mb v,\lambda}(\beta):=\mb P\mc R(\mb v,1+\beta)+\mb P\mb R(1+\beta)+\mb P\mb I_{\mathbf v,\lambda}(\beta)
					$$
				where
					$$
						\mb I_{\mb v,\lambda}(\beta):=\int_0^\infty e^{-s}\Big(\mb N(\Phi_{\mathbf v,\lambda,1+\beta}(s))+\mathbf G_{\lambda,1+\beta}\big(s,\Phi_{\mathbf v,\lambda,1+\beta}(s)\big)\Big)ds.
					$$
				Thus, Equation \eqref{inner prod is zero} is equivalent to
					$$
						\beta=\Sigma_{\mb v,\lambda}(\beta)=\tilde\kappa(\mb\Sigma_{\mb v,\lambda}(\beta)|\mathbf g)_\mc H
					$$
				for some $\tilde\kappa\in\mathbb R\setminus\{0\}$. We aim to show that $\Sigma_{\mb v,\lambda}:[-\frac{\d}{C},\frac{\d}{C}]\to[-\frac{\d}{C},\frac{\d}{C}]$ is a contraction map.
			
				Direct calculation shows that
					$$
						\Sigma_{\mb v,\lambda}(\beta)=O\Big(\frac{\d}{C^2}\Big)+O(\delta^2).
					$$
				Thus, for $C>0$ sufficiently large and $\d_0>0$ sufficiently small depending on $C$, we obtain the bound $|\Sigma_{\mb v,\lambda}|\leq\frac{\d}{C}$. To see that it is a contraction, let $\beta_1,\beta_2\in[-\frac{\d}{C},\frac{\d}{C}]$ and denote by $\Phi\in\mc X_\d$ the solution corresponding to $T_1=1+\beta_1$ and by $\Psi\in\mc X_\d$ the solution corresponding to $T_2'=1+\beta_2$ with $\lambda$ common. Inequality \eqref{Lipschitz estimate} yields
					$$
						\|\Phi-\Psi\|_\mc X\lesssim\|\Phi_0(\mathbf v, T_1)-\Phi_0(\mathbf v, T_2)\|_\mc H\lesssim|\beta_1-\beta_2|.
					$$
				Thus, Proposition \ref{locally lipschitz nonlinearity estimate} yields 
					$$
						\|\mb P\mb I_{\mathbf v,\lambda}(\beta_1)-\mb P\mb I_{\mathbf v,\lambda}(\beta_2)\|_\mc H\lesssim\d|\beta_1-\beta_2|.
					$$
				Since $\mb P\in\mc B(\mc H)$, we obtain 
					$$
						|\Sigma_{\mb v,\lambda}(\beta_1)-\Sigma_{\mb v,\lambda}(\beta_2)|\lesssim\d|\beta_1-\beta_2|.
					$$
				Upon taking $\d_0>0$ smaller if necessary, we see that $\Sigma_{\mb v,\lambda}$ is a contraction. Thus, the Banach fixed point theorem implies the existence of a unique $\beta\in[-\frac{\d}{C},\frac{\d}{C}]$ such that $\mathbf C_{\mathbf v,\lambda,T}=\mathbf 0$ with $T=1+\beta$.
			
				Now, we show that the corresponding $T$ obtained depends Lipschitz continuously on $\mathbf v$ and $\lambda$. For $\mb v,\mb w\in\mc Y$ and $\lambda,\lambda'$ satisfying the smallness assumptions, denote by $\beta_{\mb v,\lambda}$ and $\beta_{\mb w,\lambda'}$ the unique parameters obtained as above and $T_{\mb v,\lambda}=1+\beta_{\mb v,\lambda}$, $T_{\mb w,\lambda}=1+\beta_{\mb w,\lambda'}$. We write
					\begin{align*}
						|\beta_{\mb v,\lambda}-\beta_{\mb w,\lambda'}|=&|\Sigma_{\mb v,\lambda}(\beta_{\mb v,\lambda})-\Sigma_{\mb w,\lambda'}(\beta_{\mb w,\lambda'})|
						\\
						\leq&|\Sigma_{\mb v,\lambda}(\beta_{\mb v,\lambda})-\Sigma_{\mb v,\lambda}(\beta_{\mb w,\lambda'})|+|\Sigma_{\mb v,\lambda}(\beta_{\mb w,\lambda'})-\Sigma_{\mb w,\lambda'}(\beta_{\mb w,\lambda'})|
						\\
						\lesssim&\delta|\beta_{\mb v,\lambda}-\beta_{\mb w,\lambda'}|+|\Sigma_{\mb v,\lambda}(\beta_{\mb w,\lambda'})-\Sigma_{\mb w,\lambda'}(\beta_{\mb w,\lambda'})|.
					\end{align*}
				For the second term, we have 
					\begin{align*}
						|\Sigma_{\mb v,\lambda}(\beta_{\mb w,\lambda'})-\Sigma_{\mb w,\lambda'}(\beta_{\mb w,\lambda'})|&\lesssim\|\mc R(\mb v,T_{\mb w,\lambda})-\mc R(\mb w,T_{\mb w,\lambda})\|_\mc H+\|\mb I_{\mb v,\lambda}(\beta_{\mb w,\lambda})-\mb I_{\mb w,\lambda'}(\beta_{\mb w,\lambda})\|_\mc H
						\\
						&\lesssim\|\mb v-\mb w\|_\mc Y+|\lambda-\lambda'|.
					\end{align*}
				By taking $\d_0>0$ smaller if necessary, we obtain the desired Lipschitz dependence.
			\end{proof}
	
	We now show that the solution obtained in Lemma \ref{correction is zero} is a classical solution.
			\begin{proposition} \label{classical soln}
				Let $\d_0>0$ and $C>0$ be as in Lemma \ref{correction is zero}, $\d\in(0,\d_0]$, and $\mb v\in\mc Y$ such that $\|\mb v\|_{\mc Y}\leq\frac{\d}{C^2}$. Then the unique solution $\Phi$ of Equation \eqref{duhamel with id operator} belongs to $C^2([0,\infty)\times\mathbb B^7)\times C^1([0,\infty)\times\mathbb B^7)$ and solves Equation \eqref{abstract ivp} classically.
			\end{proposition}
			\begin{proof}
				For $\mathbf v$ as stated, denote by $T$ and $\Phi$ the unique parameter and solution of Equation \eqref{duhamel with id operator}. In particular, recall that $\Phi\in C([0,\infty),\mc H)$
				
				A standard fixed point argument, i.e. Theorem 1.4 of \cite{P83}, pg. 185, yields a unique local solution of Equation \eqref{duhamel with id operator} in $\mathcal H^6$. By uniqueness, this solution is precisely the global solution of Equation \eqref{duhamel with id operator} in $\mathcal H$ from Lemma \ref{correction is zero} on its interval of existence. Thus, we have $\Phi\in C([0,\mc T],\mc H^6)$ with $\mc T>0$ denoting the lifespan of the solution $\Phi$ in $\mc H^6$. Moreover, by Theorem 1.4 of \cite{P83}, this local solution is in fact global in $\mathcal H^6$ so long as $\lim_{\tau\to\mathcal T^-}\|\Phi(\tau)\|_{\mathcal H^6}<\infty$. From Equation \eqref{duhamel with id operator} and Propositions \ref{locally lipschitz nonlinearity estimate} and \ref{lipschitz nonlinearity estimate}, it follows that
					$$
						\|\Phi(\t)\|_{\mathcal H^6}\lesssim1+\int_0^\t\|\Phi(s)\|_{\mathcal H^6}ds
					$$
				for all $\t\in[0,\mc T]$. Gr\"onwall's inequality then implies $\|\Phi(\t)\|_{\mc H^6}\leq C_1e^{C_2\mc T}$ for all $\t\in[0,\mc T]$ and for some $C_1,C_2>0$. Thus, it must hold that $\mc T=\infty$. In particular, Sobolev embedding yields $\Phi(\t)\in C^2(\overline{\mathbb B^7})\times C^1(\overline{\mathbb B^7})$ for all $\t\geq0$. 
				
				To prove regularity in $\t$, we first note that $\Phi_0(\mathbf v, T)\in\mathcal D(\mathbf L)$. Theorem 1.5 of \cite{P83} implies that $\Phi$ is a classical solution, i.e., 
					$$
						\Phi\in C([0,\infty),\mathcal H^6)\cap C^1((0,\infty),\mathcal H^{6}),\quad\Phi(\tau)\in\mathcal D(\mathbf L),\;\tau\geq0
					$$
				and solves
					\begin{equation}
					\begin{cases}
						\partial_\t\Phi(\t)=\mathbf L\Phi(\t)+\mathbf N(\Phi(\t))+\mathbf G_{\lambda,T}(\tau,\Phi(\t)) \label{classical eqn_}
						\\
						\Phi(0)=\Phi_0(\mathbf v,T)
					\end{cases}
					\end{equation}
				for $\t>0$. In particular, $\partial_\tau\Phi(\tau)\in \mathcal H^6\hookrightarrow C^2(\overline{\mathbb B^7})\times C^1(\overline{\mathbb B^7})$ for $\tau>0$. Furthermore, the embedding $\mathcal H^6\hookrightarrow L^\infty(\mathbb B^7)\times L^\infty(\mathbb B^7)$ implies that Equation \eqref{classical eqn_} holds pointwise on $\mathbb B^7$. By a generalized version of Schwarz' theorem (see e.g. Theorem 9.41 on p. 235 of \cite{R76}), we can exchange $\t$-derivatives and $\xi$-derivatives upon which the claim follows.
			\end{proof}

	\subsection{Proof of the main result}\label{Proof of the main result}
		\begin{proof}[Proof of Theorem \ref{main result}]
			Let $\delta_0>0$ be sufficiently small as in Lemma \ref{correction is zero}. Let $\d\in(0,\d_0]$ and $C>0$ be as in Lemma \ref{correction is zero}. Let $\lambda\in(0,\d]$, and set $\d':=\frac{\d}{C}$. Furthermore, let $(f,g)\in H^6(\mathbb B_{2}^7)\times H^5(\mathbb B_{2}^7)$ satisfy
				$$
					\big\|\big(f,g\big)\big\|_{H^6(\mathbb B^7_{2})\times H^5(\mathbb B^7_{2})}\leq\frac{\d'}{C}.
				$$
			Then $\mathbf v:=(f,g)$ satisfies the hypotheses of Lemma \ref{correction is zero} and Proposition \ref{classical soln}. Thus, there is a unique $T\in[1-\d',1+\d']$ depending Lipschitz continuously on $\mathbf v$ and $\lambda$ so that Equation \eqref{duhamel with id operator} has the unique classical solution $\Phi=(\varphi_1,\varphi_2)\in C^2([0,\infty)\times\overline{\mathbb B^7})\times C^1([0,\infty)\times\overline{\mathbb B^7})$ with $\Phi\in\mathcal X_{\d'}$. Now, set
				\begin{align*}
					u(t,r):=&\frac{1}{T-t}\bigg[\tilde U\Big(\frac{r}{T-t}\Big)+\varphi\Big(t,\frac{r}{T-t}\Big)\bigg]
				\end{align*}
			with
				$$
					\varphi\Big(t,\frac{r}{T-t}\Big):=\varphi_1\bigg(\log\Big(\frac{T}{T-t}\Big),\frac{r}{T-t}\bigg).
				$$
			By Proposition \ref{classical soln} and the fact that similarity coordinates define a diffeomorphism of the backwards light cone into the infinite cylinder, we have that $u\in C_\text{rad}^2(\mathfrak C_T)$. Furthermore, according to Proposition \ref{correction is zero} and the calculations carried out in Section \ref{First-order formulation}, $u$ is indeed the unique solution of Equation \eqref{semilinear skyrme eqn} on $\mathfrak C_T$ satisfying the initial conditions
				$$
					u(0,r)=\tilde U(r)+f(r)
				$$
			and
				$$
					\partial_tu(0,r)=(1+\Lambda)\tilde U(r)+g(r).
				$$
			Lipschitz dependence on $(f,g)$ and $\lambda$ follows from that of $\Phi$ on $\mathbf v$ and $\lambda$ respectively. Finally, the decay estimate \eqref{convergence} follows from $\Phi\in\mathcal X_{\d'}$.
		\end{proof}

\bibliographystyle{plain}
\bibliography{bibfile}

\begin{thebibliography}{10}

\bibitem{BDS19}
Pawe{\l} Biernat, Roland Donninger, and Birgit Sch{\"o}rkhuber.
\newblock Hyperboloidal similarity coordinates and a globally stable blowup
  profile for supercritical wave maps.
\newblock {\em International Mathematics Research Notices},
  2021(21):16530--16591, 2019.

\bibitem{BCR07}
Piotr Bizon, Andrzej Rostworowski, and Tadeusz Chmaj.
\newblock {Asymptotic stability of the Skyrmion}.
\newblock {\em Physical Review. D, Particles Fields}, 75(12), 6 2007.

\bibitem{CMS23}
Po-Ning Chen, Michael McNulty, and Birgit Sch{\"o}rkhuber.
\newblock Singularity formation for the higher dimensional {Skyrme} model in
  the strong field limit.
\newblock {\em arXiv e-prints}, 2310.07042, 2023.

\bibitem{C13}
Matthew Creek.
\newblock {Large-data global well-posedness for the $1+2$-dimensional
  equivariant Faddeev model}.
\newblock {\em arXiv e-prints}, 1310.4708, 2013.

\bibitem{CDSS16}
Matthew Creek, Roland Donninger, Wilhelm Schlag, and Stanley Snelson.
\newblock Linear stability of the {Skyrmion}.
\newblock {\em International Mathematics Research Notices}, 2017(8):2497--2537,
  2016.

\bibitem{D23}
Roland Donninger.
\newblock Spectral theory and self-similar blowup in wave equations.
\newblock {\em arXiv e-prints}, 2310.12016, 2023.

\bibitem{GDs13}
Dan-Andrei Geba and Daniel Da~Silva.
\newblock On the regularity of the $2 + 1$ dimensional equivariant {Skyrme}
  model.
\newblock {\em Proceedings of the American Mathematical Society},
  141(6):2105--2115, 2013.

\bibitem{GG17}
Dan-Andrei Geba and Manoussos Grillakis.
\newblock {Large data global regularity for the $2+1$-dimensional equivariant
  Faddeev model}.
\newblock {\em Differential and Integral Equations}, 32, 08 2017.

\bibitem{GG18}
Dan-Andrei Geba and Manoussos~G Grillakis.
\newblock Large data global regularity for the classical equivariant {Skyrme}
  model.
\newblock {\em Discrete and Continuous Dynamical Systems}, 38(11):5537--5576,
  2018.

\bibitem{GNR11}
Dan-Andrei Geba, Kenji Nakanishi, and Sarada Rajeev.
\newblock Global well-posedness and scattering for {Skyrme} wave maps.
\newblock {\em Communications on Pure and Applied Analysis}, 11, 06 2011.

\bibitem{GNZ15}
Dan-Andrei Geba, Kenji Nakanishi, and Xiang Zhang.
\newblock {Sharp global regularity for the $2 + 1$-dimensional equivariant
  Faddeev model}.
\newblock {\em International Mathematics Research Notices},
  2015(22):11549--11565, 02 2015.

\bibitem{GL60}
M.~Gell-Mann and M.~L{\'e}vy.
\newblock The axial vector current in beta decay.
\newblock {\em Il Nuovo Cimento}, 16(4):705--726, 1960.

\bibitem{G22a}
Irfan Glogi{\'c}.
\newblock {Stable blowup for the supercritical hyperbolic Yang-Mills
  equations}.
\newblock {\em Advances in Mathematics}, 408:108633, 2022.

\bibitem{KL83}
L.~V. Kapitanskii and O.~A. Ladyzhenskaya.
\newblock The {Coleman} principle for finding stationary points of invariant
  functionals.
\newblock {\em Zap. Nauchn. Sem. Leningrad. Otdel. Mat. Inst. Steklov. (LOMI)},
  15(127):84--102, 1983.

\bibitem{L21}
Dong Li.
\newblock {Global well-posedness of hedgehog solutions for the (3+1) Skyrme
  model}.
\newblock {\em Duke Math. J.}, 170(7):1377--1418, 2021.

\bibitem{MS04}
Nicholas Manton and Paul Sutcliffe.
\newblock {\em Topological Solitons}.
\newblock Cambridge Monographs on Mathematical Physics. Cambridge University
  Press, 2004.

\bibitem{MT91}
J.~B. McLeod and W.~C. Troy.
\newblock The {Skyrme} model for nucleons under spherical symmetry.
\newblock {\em Proceedings of the Royal Society of Edinburgh Section A:
  Mathematics}, 118(3-4):271--288, 1991.

\bibitem{M20}
Michael McNulty.
\newblock Development of singularities of the {Skyrme} model.
\newblock {\em Journal of Hyperbolic Differential Equations}, 17(01):61--73,
  2020.

\bibitem{P83}
A.~Pazy.
\newblock {\em Semigroups of Linear Operators and Applications to Partial
  Differential Equations}.
\newblock Applied Mathematical Sciences. Springer New York, NY, 1 edition,
  1983.

\bibitem{R12}
Jeffrey Rauch.
\newblock {\em Hyperbolic partial differential equations and geometric optics},
  volume 133 of {\em Graduate Studies in Mathematics}.
\newblock American Mathematical Society, Providence, RI, 2012.

\bibitem{R76}
Walter Rudin.
\newblock {\em Principles of Mathematical Analysis}.
\newblock International Series in Pure and Applied Math- International Series
  in Pure and Applied Mathematics. McGraw-Hill Book Co., New
  York-Auckland-D{\"u}sseldorf, 3rd edition, 1976.

\bibitem{S88}
Jalal Shatah.
\newblock {Weak solutions and development of singularities of the $SU(2)$
  $\sigma$-model}.
\newblock {\em Communications on Pure and Applied Mathematics}, 41(4):459--469,
  1988.

\bibitem{S61a}
T.H.R. Skyrme.
\newblock A non-linear field theory.
\newblock {\em Proceedings of the Royal Society of London. Series A.
  Mathematical and Physical Sciences}, 260(1300):127--138, 1961.

\bibitem{S61b}
T.H.R. Skyrme.
\newblock Particle states of a quantized meson field.
\newblock {\em Proceedings of the Royal Society of London. Series A.
  Mathematical and Physical Sciences}, 262(1309):237--245, 1961.

\bibitem{S62}
T.H.R. Skyrme.
\newblock A unified field theory of mesons and baryons.
\newblock {\em Nuclear Physics}, 31:556--569, 1962.

\bibitem{W11}
Willie Wai-Yeung Wong.
\newblock Regular hyperbolicity, dominant energy condition and causality for
  {Lagrangian} theories of maps.
\newblock {\em Classical and Quantum Gravity}, 28(21):215008, sep 2011.

\end{thebibliography}

\end{document}